\documentclass[a4paper,11pt]{amsart}
\usepackage[matrix,arrow,tips,curve]{xy}
\usepackage[english]{babel}
\usepackage[leqno]{amsmath}
\usepackage{amssymb,amsthm}
\usepackage{amscd}
\usepackage{enumerate}

\input{xy}
\xyoption{all}

\newtheorem{thm}{Theorem}[subsection]
\newtheorem{lemma}[thm]{Lemma}
\newtheorem{lemmadefi}[thm]{Lemma - Definition}

\newtheorem{prop}[thm]{Proposition}

\newtheorem{claim}[thm]{Claim}

\newtheorem{fact}[thm]{Fact}

\theoremstyle{definition}
\newtheorem{defi}[thm]{Definition}
\newtheorem{nota}[thm]{}
\newtheorem{remark}[thm]{Remark}
\numberwithin{equation}{section}

\newtheorem{example}[thm]{Example}

\newcommand{\la}{\longrightarrow}
\newcommand{\ha}{\hookrightarrow}

\newcommand{\ov}{\overline}

\newcommand{\Aut}{\operatorname{Aut}}
\newcommand{\Out}{\operatorname{Out}}

\newcommand{\Z}{\mathbb{Z}}
\newcommand{\R}{\mathbb{R}}

\newcommand{\G}{\Gamma }

\def\mo{\underline{0}}

\def\Mgb{\overline{M}_g}

\newcommand{\Mgpnn}{M_{g,n}^{\rm pure}}

\newcommand{\Mgtn}{{M_{g,n}^{\rm trop}}}

\newcommand{\Mgrn}{M_{g,n}^{\rm reg}}
\newcommand{\Mgp}{M_g^{\rm pure}}

\newcommand{\Mgt}{{M_g^{\rm trop}}}
\newcommand{\Mge}{M_g^{\rm trop}[3]}
\newcommand{\Mgpe}{M_g^{\rm pure}[3]}
\newcommand{\Mgre}{M_g^{\rm reg}[3]}
\newcommand{\Mgr}{M_g^{\rm reg}}
\newcommand{\Agt}{A_g^{\rm trop}}
\newcommand{\Sgt}{{\rm Sch}_g^{\rm trop}}
\newcommand{\Sgp}{\rm{Sch}_g^{\rm pure}}

\newcommand{\tgt}{t_g^{\rm trop}}
\newcommand{\tgp}{t_g^{\rm pure}}

\begin{document}

\title{Geometry of tropical moduli spaces and linkage of  graphs}
\author{Lucia Caporaso}
\address{Dipartimento di Matematica,
Universit\`a Roma Tre,
Largo S. Leonardo Murialdo 1,
00146 Roma (Italy)}
\email{caporaso@mat.uniroma3.it}
%
\maketitle
 \begin{abstract}
 We prove the following ``linkage" theorem:  two $p$-regular graphs of the same genus can be obtained from one another by a finite alternating sequence of one-edge-contractions; moreover  this preserves
  3-edge-connectivity.  
We use the linkage theorem to prove that various moduli spaces of tropical curves 
 are connected through codimension one.
  \end{abstract}

\tableofcontents

\section{Introduction}

This paper is made of two parts, with the second partially motivating the first.
The second part  studies the moduli spaces of tropical curves;
in order to establish some remarkable connectedness properties, we encounter some questions about graphs
which are of interest in their own right. 
The solution of these graph theoretic problems occupies the first part of this paper.

Let us describe the two parts  in some details. The first is concerned with classification of $p$-regular (every vertex has valency $p$) connected graphs. It is quite easy to see that there exists a unique $1$-regular graph, namely two vertices joined by a unique edge. Similarly,
$2$-regular graphs are classified by the number of their edges, indeed
there exists a unique $2$-regular graph with $n$-edges: the cycle on $n$ vertices,
and these are all the 2-regular graphs.
As soon as $p\geq 3$ the situation gets complicated;  as a matter of fact,
as far as we are aware of, the number of $3$-regular graphs with fixed first Betti number
is not known.
And this number would be very interesting  for several reasons; for instance, it counts the
$0$-dimensional combinatorial cycles in the moduli space of Deligne-Mumford stable curves, $\Mgb$.
See \cite{art} for more   on this  issue.

Our main result in the first part of the paper is Theorem~\ref{main}.
This states, first of all, that any two $p$-regular connected  graphs
  $\Gamma$ and $\Gamma '$,  with the same first Betti number,
are ``linked", i.e. they can be obtained one from the other with a finite sequence of alternating one-edge contractions as follows.
There exists a finite sequence
\begin{equation}
\xymatrix{
\Gamma=\Gamma_1 \ar[rd] ^{ }&&\Gamma_3\ar[ld] _{}\ar[rd]^{}&\ldots &  \ldots&\ar[ld]_{}\Gamma_{2h+1}=\Gamma ' \\
 & \Gamma_2   
 &  &\ldots& \Gamma_{2h}&\\
}
\end{equation}
where every arrow is the  map  contracting precisely one edge and leaving everything else
unchanged. Also, every odd-indexed graph
in the diagram above is $p$-regular.
Secondly, we prove that if $\Gamma$ and $\Gamma'$ are 3-edge-connected there exists a diagram as above where   the   graph  $\Gamma_i$ is  
3-edge-connected, for every $i=1,\ldots, 2h+1$. This second part makes the proof seriously more complicated, but it does play an important  role in the application of this result to the second part of the paper.
We refer to this property as the ``conservation of 3-edge-connectivity".

In  case $p=3$  the result, without the conservation of 3-edge-conectivity, is due to A.Hatcher and W.Thurston 
\cite{HT}, by a  non combinatorial argument;
a combinatorial proof valid for simple graphs is given by Y.Tsukui  \cite{T}.

Our proof is purely combinatorial. We first reduce it to hamiltonian graphs (in Subsection~\ref{hamsec}), and then show that every hamiltonian  graph is linked to a special type of graph  called the $p$-polygon (in Subsection~\ref{linksec}).

Now we turn to the part concerning moduli of tropical curves, which occupies Section~\ref{tropsec}. 
The moduli space, $\Mgt$, of tropical curves  of genus $g$,  
and the moduli space of $n$-pointed tropical curves, $\Mgtn$,
are here treated simply as topological spaces. The point is, their geometry is   so complex   that they don't look like  tropical varieties (the case $g=0$ is an exception); in fact the problem to find a  ``good" category in which they should be placed is under investigation, and still awaits to be resolved.
In a similar vein, it is interesting to study which topological properties of tropical varieties are also valid for those moduli spaces.  

One of the characterizing properties of tropical varieties (defined by prime ideals)
is that they are ``connected through codimension one"; see the  Structure Theorem in \cite[Ch. 3]{MS}.
The goal of Section~\ref{tropsec} is thus to establish that several moduli spaces of tropical curves are connected through codimension one; our motivating observation  was that this property
is strictly  related to the linkage properties of graphs studied in the first part of the paper.

Let us now give more details.
In this paper,
together with
the original notion of tropical curve,  here called
``pure tropical curve",  due to G. Mikhalkin  (see \cite{MIK2}),  
we use the generalization     given  by S. Brannetti, M. Melo and F. Viviani in \cite{BMV};
the advantage of the generalized notion is that, with it, 
the moduli space is closed under specialization, while the moduli space of pure tropical curves is not
(see \cite{BMV} or \cite{CHBK} for details). 

We are interested in the spaces $\Mgtn$, and also in
the Schottky
space, $\Sgt$, defined as the quotient of $\Mgt$ via the Torelli map, 
studied in \cite{BMV} and \cite{chan}.
They are   easily   seen to be connected, however a stronger property holds, as we are going to explain.
Let us focus on $\Mgt$ for simplicity; we have  a finite decomposition
$\Mgt=\sqcup_{i\in I} M_i$ where each $M_i$ is  
   a connected orbifold,
and  $\Mgt$ is the closure of the union of those $M_i$ having maximal dimension (equal to $3g-3$).
Every $M_i$ has a clear geometric interpretation, for example
the above mentioned dense union
$$
\Mgr:=\bigsqcup_{i\in I: \dim M_i=3g-3}M_i\subset \Mgt=\bigsqcup_{i\in I}M_i
$$
parametrizes genus-$g$ tropical curves whose underlying graph is $3$-regular.
Moreover $\Mgr$ is open in $\Mgt$. 
Now, $\Mgr$ is clearly not connected, whereas
  $\Mgt$ is so, therefore one can  ask:
if we add to $\Mgr$ all strata $M_i$ of  codimension one  (i.e. of dimension $3g-4$), do we get a connected space?
Equivalently: is $\Mgt$ connected through codimension one?

The answer to the question is yes, 
and, as we said, this follows from the linkage theorem for graphs.
In fact, by  Proposition~\ref{connt1} , connectedness through codimension one holds   for  all $\Mgtn$. The proof  is based on an extension of Theorem~\ref{main} for $p=3$ to graphs with legs, Proposition~\ref{linkn}.

Next, by the tropical Torelli theorem   of \cite{CV}, and its generalization in \cite{BMV}, 
 the  Schottky  locus  $\Sgt$ 
 is the image via the Torelli map of the locus in $\Mgt$
parametrizing 3-edge-connected tropical curves.
This motivates our interest    in 3-edge-connected graphs.
Indeed, the fact that graph linkage   preserves 3-edge-connectivity, enables us to prove  that $\Sgt$ is connected through codimension one;
see Theorem~\ref{conn1}.

{\it Acknowledgments.} I am grateful to Margarida Melo and Filippo Viviani for their precious remarks on  this paper and its previous versions, and to the referees for correcting several inaccuracies.

\section{The linkage theorem}
\subsection{Terminology.}
 Throughout the paper, $p, g$ and $n$ will be integers with $p\geq 3$ and $g,n\geq 0$.
 
 $\Gamma$ always denotes  a   graph 
(i.e. a one dimensional finite simplicial  complex),
  $V(\Gamma)$
 the set of its vertices (or $0$-cells)  and    $E(\Gamma)$ the set of its edges (or $1$-cells). 
Every $e\in E(\Gamma)$ joins two, possibly equal, vertices,   called
the {\it endpoints} of $e$. If the two endpoints of $e$ coincide we say that $e$ is a {\it loop}.

We assume 
all graphs   to be  connected, unless we specify otherwise.  The combinatorial definition of   graph is in Definition~\ref{gdef}.

The first Betti number, or the genus,  of $\Gamma$ is
$ 
b_1(\Gamma)=|E(\Gamma)| - |V(\Gamma)| + c 
$ 
 where $c$ is the number of connected components of $\Gamma$.

Using the standard terminology (see \cite{Die})  a graph $\Gamma$ is  called
\begin{enumerate}
\item
 $p${\it -regular} if every vertex has valency (or degree) equal to $p$;
\item
a {\it path} if its first Betti number is equal to 0, and if it contains no vertex of valency 
$\geq 3$. A path $\Gamma$ satisfies $|V(\Gamma)|=|E(\Gamma)|+1$;
we shall say that $|E(\Gamma)|$ is the length of the path.
\item
a {\it cycle} if it is   2-regular. A cycle     has 
$b_1(\Gamma)=1$, and hence an equal number of edges and vertices; this number will be called its   length.
 \item
 {\it $p$-edge-connected} if $|V(\Gamma)|\geq 1$ and if $\Gamma\smallsetminus F$ is connected for any $F\subset E(\Gamma)$ with $|F|< p$.
\end{enumerate}

\begin{nota}{\it Contractions and  linkage.}
\label{cont}
Fix $\Gamma$ and    $e\in E(\Gamma)$. Let 
$\Gamma/e$ be the graph obtained by contracting $e$ to a point and leaving everything else unchanged  (\cite[sect I.1.7]{Die}). Then there is a natural surjective map
  $\Gamma \to \Gamma '$, called    the {\it contraction} of $e$.
  More generally, if $S\subset E(\Gamma)$ is a set of edges, we denote by $\Gamma/S$ the contraction of every edge in $S$ and denote by  $\sigma:\Gamma\to \Gamma/S$ the associated map.
  Let $T:=E(\Gamma)\smallsetminus S$. Then there is a natural identification between
  $E(\Gamma/S)$ and $T$. Moreover   $\sigma$ induces a surjection
  $$
  \sigma_V:V(\Gamma)\la V(\Gamma/S);\  \  \  v\mapsto \sigma(v).
  $$
Notice that every connected component of 
  $\Gamma - T$ 
  (the graph obtained from $\Gamma$ by  removing every edge in $T$)
  gets contracted to a vertex of $\Gamma/S$; conversely, for every vertex $\ov{v} $ of $\Gamma/S$ its preimage $\sigma^{{-1}}(\ov{v})\subset \Gamma$ is a connected component of  $\Gamma - T$.
  In particular, we obtain the following useful identity:
  \begin{equation}
  \label{gendec}
b_1(\Gamma- T)=\sum_{\ov{v}\in V(\Gamma/S)}b_1(\sigma_V^{{-1}}(\ov{v})).
\end{equation}
  \end{nota}
\begin{remark}
\label{lme}
Let $\sigma:\Gamma\to \Gamma/S$ be the contraction of $S$ as above.
The following facts are well known and easy to prove.
\begin{enumerate}
 \item
$  b_1(\Gamma)=b_1(\Gamma/S)+b_1(\Gamma\smallsetminus T).$
  \item
 If $\Gamma$ is $p$-edge-connected so is $\Gamma/S$.
  \end{enumerate}
\end{remark}

\begin{defi}
\label{link}\begin{enumerate} 
\item
Let $\Gamma_1$ and $\Gamma _2$ be two   graphs.
We say that 
 $\Gamma_1$ and $\Gamma _2$ are {\it strongly linked} if for $i=1,2$ there exists
a non-loop edge  $e_i\in E(\Gamma_i)$
 such that
 the contraction of $e_1$ and the contraction of $e_2$ coincide, i.e.
$$
 \Gamma_1 \stackrel{\sigma_1}{\la} \Gamma_1/e_1=\Gamma_2/e_2\stackrel{\sigma_2}{\longleftarrow}\Gamma_2, 
$$
 and  
 $\sigma_1(e_1)=\sigma_2(e_2)$ (i.e. $e_1$ and $e_2$ are mapped  to  the same vertex).
 \item
Let $\Gamma$ and $\Gamma '$ be two   graphs.
We say that $\Gamma$ and $\Gamma '$ are {\it linked} if there exists a
finite sequence of graphs
$$
\Gamma =  \Gamma_1, \Gamma_2,\ldots,    \Gamma_{n-1},\Gamma_n=\Gamma '
$$
such that $\Gamma_i$ and $\Gamma_{i-1}$ are strongly linked, for every $i=2,\ldots, n$.
\end{enumerate}
  \end{defi}
  We are particularly interested in 3-edge-connected   graphs, so we need
 the following variant.
  \begin{defi}
\label{3link}
Let $\Gamma$ and $\Gamma '$ be two   3-edge-connected   graphs.
We say that $\Gamma$ and $\Gamma '$ are {\it 3-linked} 
  if there exists a
finite sequence of   3-edge-connected  graphs
$$
\Gamma=  \Gamma_1, \Gamma_2,\ldots,    \Gamma_{n-1},\Gamma_n=\Gamma '
$$
such that 
  $\Gamma_i$ and $\Gamma_{i-1}$ are strongly linked, for every $i=2,\ldots, n$.
 \end{defi}

  \begin{remark}
  \label{}
Being linked, or 3-linked  is an equivalence relation.
Linked graphs have
 the same number of edges and  vertices.
\end{remark}

\begin{example}
\label{pet}
The next picture represents two strongly linked $3$-regular graphs, with    $\Gamma_1/e_1$ 
  equal to $\Gamma_2/e_2$. $\Gamma_1$ is   called ``Petersen" graph.

 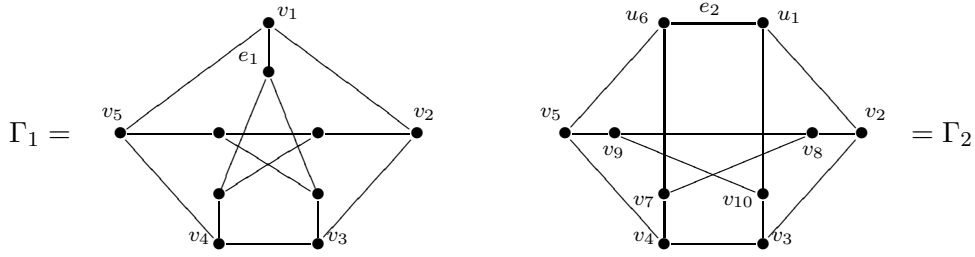
\begin{figure}[!htp]
\label{fig1}
$$\xymatrix@=1pc{
&&&&*{\bullet}&&&&&&
&&*{\bullet}&&*{\bullet}\ar@{-}[ll]_<(.5){e_2}&&&
\\
&&&&*{\bullet}\ar@{-}[u]^<(.2){e_1}_>(1.4){v_1}&&&&
\\
\Gamma_1=&*{\bullet}\ar@{-}[uurrr]^<{v_{5}}&&*{\bullet}\ar@{-}[ll]&&*{\bullet}\ar@{-}[ll]&&*{\bullet}\ar@{-}[ll]\ar@{-}[uulll]_<{v_2}&
&&*{\bullet}\ar@{-}[uurr]^>{u_{6}}^<{v_{5}}&*{\bullet}\ar@{-}[l]&&&&*{\bullet}\ar@{-}[llll]\ar@{-}[r]&*{\bullet}\ar@{-}[uull]_>{u_1}_<{v_2}&=\Gamma_2
\\
&&&*{\bullet}\ar@{-}[uur]\ar@{-}[urr]&&*{\bullet}\ar@{-}[uul]\ar@{-}[ull]&&&&
&&&*{\bullet} \ar@{-}[uuu]\ar@{-}[urrr]_>{v_8}&&*{\bullet}\ar@{-}[uuu]\ar@{-}[ulll]^>{v_9}
\\
&&&*{\bullet}\ar@{-}[uull]\ar@{-}[u]^<{v_4}&&*{\bullet}\ar@{-}[uurr]\ar@{-}[ll]\ar@{-}[u]_<{v_3}&&&
&&&&*{\bullet}\ar@{-}[uull]\ar@{-}[u]^<{v_4}^>{v_7}&&*{\bullet}\ar@{-}[uurr]\ar@{-}[ll]\ar@{-}[u]_<{v_3}^>{v_{10}}\\
}$$
\caption{Petersen graph strongly linked to a  hamiltonian graph.}
\end{figure}
\end{example}

\

 \begin{remark}
\label{count}
Let $\Gamma$ be a $p$-regular graph with $p\geq 3$; set $b=b_1(\Gamma)$. We have $|E(\Gamma)|=|V(\Gamma)|p/2$ hence
$$
b=1+\frac{(p-2)|V(\Gamma)|}{2},\  \   \  |V(\Gamma)|=\frac{2b-2}{p-2} \  \text{ and } \  |E(\Gamma)|=\frac{p(b-1)}{p-2}.
$$
If  $\Gamma$ is $3$-regular,
$|E(\Gamma)|=3b-3$ and $|V(\Gamma)|=2b-2$.
\end{remark}

\subsection{$p$-regular   hamiltonian graphs}
\label{hamsec}
\begin{defi}
\label{polydef}
A graph $\Gamma$ is called {\it hamiltonian}  if $|V(\Gamma)|\geq 2$ and if
 it contains a {\it hamiltonian cycle}, i.e.  a cycle
  passing through every vertex (exactly once).
   A $p$-regular hamiltonian graph
   free from loops is called a {\it p-hamiltonian} graph.
   \end{defi}
Examples of $p$-hamiltonian graphs are all the graphs in figures 3, 4 and 5. 
In Figure 1, the graph $\Gamma_2$ (on the right)  is hamiltonian, with $b_1(\Gamma_2)=5$,
whereas the graph $\Gamma_1$ is not hamiltonian.
So, Example~\ref{pet} implies that the non hamiltonian graph  $\Gamma_1$ is linked to the $3$-hamiltonian  graph $\Gamma_2$. This is true in general, by the following Proposition~\ref{poly}.

In the next two proofs  we will use the following  terminology.
To every edge $e$ of a graph $\Gamma$ we associate two {\it half-edges}, $h$ and $h'$, defined as follows.
Call $v,v'\in V(\Gamma)$ the two (possibly equal) endpoints of $e$. Then
$h$ and $h'$ are line segments   such that $h\subsetneq e$,  $h'\subsetneq e$,
$e=h\cup h'$ and such that
each of them contains precisely one end of $e$, so $v\in h$ and $v'\in h'$.
There are many possible choices for the half-edges of any edge, but we shall assume that such a choice is made; in fact everything we will say does not depend on this choice.
For example: the valency of   $v\in E(\Gamma)$ 
is  equal to the number of half-edges of $\Gamma$ touching $v$  (see also \ref{gdef}).

\begin{prop}
\label{poly}
Every  $p$-regular    graph 
$\Gamma$ is linked to a $p$-hamiltonian graph.
Every  $p$-regular,  3-edge-connected graph is
3-linked to a 3-edge-connected  $p$-hamiltonian graph.
\end{prop}

\begin{proof}
Let $\Gamma$ be our $p$-regular graph, and $b=b_1(\Gamma)$.
Call $\ell (\Gamma)$ the maximal length of a cycle contained in $\Gamma$;
by Remark~\ref{count} we have
     $\ell  (\Gamma) \leq |V(\Gamma)|=\frac{2b-2}{p-2}$. We shall use descending induction on $\ell (\Gamma)$. 
If $\ell  (\Gamma)=\frac{2b-2}{p-2}$ there is nothing to prove, so the basis of the induction is done.

Assume $\ell  (\Gamma)< \frac{2b-2}{p-2}$; let $\Delta\subset \Gamma$ be a  cycle of  length
$\ell=\ell  (\Gamma)$.

For consistency with
 Definition~\ref{link}  we set $\Gamma_1=\Gamma$. 
We shall  explicitly construct a $p$-regular    graph, $\Gamma_2$, strongly linked to $\Gamma$
and such that $\ell  (\Gamma_2) > \ell  (\Gamma)$.
If $\Gamma$ is 3-edge-connected so will be $\Gamma_2$.
Using the induction hypothesis on $\Gamma_2$ will suffice to complete the proof.
 Denote $V(\Delta)=\{v_1,\ldots, v_{\ell}\}\subsetneq V(\Gamma)$.
 The forthcoming construction is pictured   in Figure 2.

Pick a vertex $v\in V(\Gamma)$ such that $v\not\in V(\Delta)$
and such that there is an edge $e$ joining $v$ to one of the vertices of $\Delta$;
obviously $e\not\in E(\Delta)$.
We can assume, with no loss of generality, that the endpoints of $e$ are $v_1$ and $v$. 
Let us call $e_1$ and $e_{\ell}$ the two edges of $\Delta$  meeting at $v_1$.

Since $v$ has valency $p$, there 
are $p-1$ half-edges containing $v$ and not contained in $e$; 
 let us call them $h_1,\ldots, h_{p-1}$.
Similarly as $v_1$ has valency $p$ there 
are $p-3$ half-edges containing $v_1$ and not contained in $e$, $e_1$ or $e_{\ell}$;
we call these $h_{p},\ldots, h_{2p-4}$. 
 It is clear that  no half-edge $h_i$     lies in $\Delta$.
Consider  the contraction of $e$
$$\sigma_1:\Gamma_1=\Gamma \la \Gamma/e=\Gamma '.
$$
Clearly $w:=\sigma_1(e)$ is a vertex of valency $2p-2$, 
indeed
the images via  $\sigma_1$ of $e_1, e_{\ell},  h_1, \ldots h_{2p-4}$
all touch $w$, and there is no other edge   touching $w$.

Now we perform a  valency reducing extension  on $w$ (cf. \cite[A.2.2]{CV}); 
namely we introduce an edge contracting map $ 
\sigma _2:\Gamma _2\la \Gamma '
$  from a new graph $\Gamma_2$
such that $ \Gamma '$ is obtained from $\Gamma_2$ as the contraction  to $w$ of a unique edge, which we call
$e_{\ell+1}$; hence $\sigma_2(e_{\ell+1})=w$ and $\sigma _2$ leaves everything else unchanged.   The two endpoints of $e_{\ell+1}$ are two vertices of
valency $p$, which we call $u_{\ell+1},u_1 \in V(\Gamma _2)$.
In $\Gamma_2$ we distribute the $2p-2$ half-edges $e_1, e_{\ell}, h_1,\ldots ,h_{2p-4}$
so that $p-1$ of them  touch $u_1$ and 
the remaining $p-1$ touch   $u_{\ell+1}$. Moreover we have the (old) edge 
$e_1$ touching
  $u_{1}$, the old edge $e_{\ell}$ touching $u_{\ell+1}$,
  and the new edge
  $e_{\ell+1}$  joining $u_1$ with $u_{\ell+1}$.
Therefore the graph $\Gamma_2$ is $p$-regular.
Summarizing, we have
$$
\Gamma _2\stackrel{\sigma _2}{\la} \Gamma _2/e_{\ell+1}=\Gamma/e \stackrel{\sigma _1}{\longleftarrow} \Gamma.
$$
Therefore $\Gamma$ and $\Gamma_2$ are strongly linked.
Now the given cycle $\Delta\subset \Gamma$ is mapped to a cycle of the same length by $\sigma_1$,
whereas $\sigma_2^{-1}(\sigma_1(\Delta))$ is a cycle of length at least $\ell +1$
(as it contains the vertices $\{u_1, v_2,\ldots, v_{\ell}, u_{\ell+1}\}$).
Therefore $\ell(\Gamma)<\ell (\Gamma _2)$.
It is clear that by iterating this construction we arrive at a $p$-regular graph 
$\widehat{\Gamma}$ with $\ell(\widehat{\Gamma})=(2b-2)/(p-2)$, so that $\widehat{\Gamma}$ is   hamiltonian graph. It is also clear that $\widehat{\Gamma} $ and $\Gamma$ are linked.

Now  suppose that $\Gamma$ is 
 3-edge-connected; then  $\Gamma'$ is also 
3-edge-connected by Remark~\ref{lme}. 
To prove that 
$\widehat{\Gamma}$ is     3-edge-connected we need to prove that
the extension of $w$ used during the proof may be constructed so as to yield  a 3-edge-connected
graph $\Gamma_2$.  This follows from the proof of \cite[Prop. A.2.4]{CV}, with trivial modifications.
Finally, by the next lemma~\ref{noloops},
  we can take $\widehat{\Gamma}$ free from loops.
\end{proof}
The next  picture  illustrates this proof. We represent the relevant portions of $\Gamma$, on the left, of $\Gamma'$  and of $\Gamma_2$.
The vertices $v_2$ and $v_{\ell}$ belong to the cycle $\Delta$, hence they are joined by a path 
(not drawn)  not intersecting $h_1$ and $h_2$. 

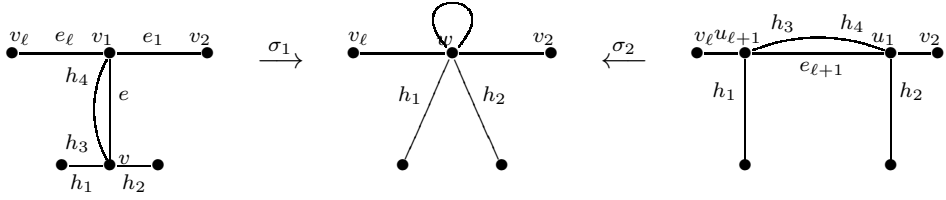
\begin{figure}[!htp]
\label{incr}
$$\xymatrix@=1pc{
*{\bullet}\ar@{-}[rr]^>{v_{1}}^<{v_{\ell}} ^<(0.5){e_{\ell}} &&*{\bullet}  \ar@{-}@/_/[dd]_<(0.2){h_4}_>(0.8){h_3} \ar@{-}[dd]^>(0.4){e}^> {v}\ar@{-}[rr]^>{v_{2}}^<(0.4){e_1} &&*{\bullet}&\stackrel{\sigma_1}{\la}&
*{\bullet}\ar@{-}[rr]^>{w}&&*{\bullet} \ar@{-}@(ul,ur)\ar@{-}[l]_>(2,3){v_{\ell}} \ar@{-}[rr]^>{v_{2}}&&*{\bullet}&\stackrel{\sigma_2}{\longleftarrow}&
*{\bullet}\ar@{-}[r]^>{u_{\ell +1}}&*{\bullet} \ar@{-}@/^/[rrr]^>(0.3){h_3}^>(1.8){h_4}\ar@{-}[l]_>{v_{\ell}}&&&*{\bullet}\ar@{-}[lll]_<{u_1}^>(0.5){e_{\ell +1}} \ar@{-}[r]^>{v_{2}}&*{\bullet}
\\
\\
&*{\bullet}\ar@{-}[r]_>(0.5){h_1}&*{\bullet}\ar@{-}[l]_>{} \ar@{-}[r]_>(0.6){h_2}&*{\bullet}&&&
&*{\bullet}\ar@{-}[ruu]^>(0.6){h_1}&&*{\bullet}\ar@{-}[luu]_>(0.6){h_2}&&&
&*{\bullet}\ar@{-}[uu]^>(0.7){h_1}&&&*{\bullet}\ar@{-}[uu]_>(0.7){h_2}
}$$
\caption{Increasing the length of $\Delta$ in the proof of \ref{poly}.}
\end{figure}

\newpage
\begin{lemma}
\label{noloops}
Every $p$-regular hamiltonian   graph
is  linked   to a  $p$-hamiltonian graph.

Every $p$-regular hamiltonian  3-edge-connected  graph
is    3-linked  to a  $p$-hamiltonian  3-edge-connected graph.
\end{lemma}
\begin{proof}
It suffices to exhibit a procedure which decreases the number of loops, preserving the property
of being hamiltonian, $p$-regular and 3-edge-connected.

Let $\Delta\subset \Gamma$ be a fixed hamiltonian cycle;
denote by $E(\Delta)=\{e_1,\ldots, e_t\}$ and $V(\Gamma)=\{v_1,\ldots, v_t\}$
with $e_i$ joining $v_i$ and $v_{i+1}$ as usual.
Suppose that $\Gamma$ contains a loop $\ell$, and assume (with no  loss of generality)
that this loop is based at $v_1$. Let $\ell_1$ and $\ell_2$ be the half-edges of the loop
(so $\ell_1$ and $\ell_2$ touch $v_1$). We know that $v_1$ is connected to $v_2$ by the edge $e_1$. 
Since $v_2$ has valency $p\geq 3$ there is a half-edge $h$ touching $v_2$, not contained in 
the hamiltonian cycle $\Delta$, and not contained in an edge touching $v_1$
(for otherwise $v_2$ would have valency less than that of $v_1$).
Let us consider $\Gamma /e_1$, and call $w$ the vertex into which 
$e_1$ is contracted. The valency of $w$ is $2p-2$.

Let $\Gamma_2$ be the graph obtained from $\Gamma$ by 
changing the loop $\ell$ into an edge, called $f_1$, joining $v_1$ with $v_2$,
and by changing the half-edge $h$ into a half-edge touching $v_1$.
This operation does not create any new loop (as the edge of $\Gamma$ containing $h$ does not touch $v_1$), and eliminates the loop $\ell$. So the number of loops of $\Gamma_2$ is less than that of $\Gamma$.
It is clear that $\Gamma_2$ is $p$-regular (we added and removed a half-edge from $v_1$ and $v_2$, and left everything else unchanged).
The hamiltonian cycle $\Delta$ is  clearly contained in $\Gamma_2$, so $\Gamma_2$
is hamiltonian.
Finally, $\Gamma/e_1=\Gamma_2/e_1$, so $\Gamma$ and $\Gamma_2$ are strongly linked.

It remains to show that if $\Gamma$ is 3-edge-connected so is  $\Gamma_2$.
This follows from \cite[Prop.A.2.4]{CV}, in fact
 the extension of $w$ given by $\Gamma_2\to \Gamma/e_1$  is the same as in Step 1 in the proof of that proposition   (with obvious modifications).
\end{proof}

\subsection{$p$-polygons}

\begin{nota}
\label{norm}
{\it Fixing a hamiltonian cycle in a   $p$-hamiltonian graph.}
We now introduce some useful conventions.
Let $\Gamma$ be a $p$-hamiltonian graph, with $b:=b_1(\Gamma)$. 
We fix a hamiltonian cycle, $\Delta$, and refer to it as the {\it distinguished} hamiltonian cycle; let $\gamma=|V(\Gamma)|$ be the length of $\Delta$.
The choice of $\Delta$ enables us to use the   following terminology.
The     edges of $\Gamma$ which do not lie in $\Delta$ 
will be called   {\emph{chords}}. 
The number of chords of $\Gamma$ is easily computed:
\begin{equation}
\label{chords}
{\text{Number of chords of }} \Gamma = |E(\Gamma)| - \gamma=\frac{p(b-1)}{p-2}-
\frac{2(b-1)}{p-2}=b-1.
\end{equation}

The vertices of $\Gamma$ will be labeled according to the cyclic structure of $\Delta$, i.e.
$V(\Gamma )=V(\Delta )=\{v_1,v_2,\ldots, v_{\gamma}\}$
so that there exists an edge $e_i\in E(\Delta)\subset E(\Gamma)$ joining $v_i$ with $v_{i+1}$ for every $i=1,\ldots, \gamma$
(with the cyclic convention $v_{\gamma+1}=v_1$); hence $E(\Delta)=\{e_1,\ldots, e_{\gamma}\}$.
The starting vertex $v_1$ can be picked arbitrarily;
furthermore, for any choice of $v_1$, there are two cyclic labelings of the vertices
(corresponding to the two cyclic orientations of $\Delta$).
Once a distinguished cycle $\Delta$ is chosen, we shall always use  such a labeling.

Let $\Gamma$ be a $p$-hamiltonian graph where a distinguished cycle $\Delta$ has been fixed.
Every chord has two distinct endpoints  ($\Gamma$ being free from loops).
We shall denote by $d_{i,j}$ a chord joining $v_i$ with $v_j$, and always assume $i<j$.
If $p\geq 4$ there may be more than one chord joining $v_i$ with $v_j$;
if we need to distinguish between them we will use superscripts, i.e. we denote
$\{d_{i,j}^{\alpha},\  \alpha=1,\ldots, m\}$ the chords  joining $v_i$ and $v_j$; notice that  $m\leq p-2$.

We also need a notation for a chord of which  only one end is known.
So, the chord having one end at the vertex $v_j$ and the other end at some other vertex
will be denoted $d_{j,*}$.

Let $d_{i,j}$ be a chord as above. 
Then $d_{i,j}$  determines two paths of the cycle $\Delta$, namely the two paths $\Lambda$ and $\Lambda '$ contained in $\Delta$,
having extremes $v_i$ and $v_j$. Hence $\Lambda\cap \Lambda ' =\{v_i,v_j\}$
and $\Lambda\cup \Lambda ' =\Delta$.
We call such two paths the {\it sides} of $d_{i,j}$.
It is obvious that one of them 
has length $j-i$
 and the other has length $\gamma-j+i$.
We define the  {\it amplitude},  $\alpha(d_{i,j})$,  of $d_{i,j}$ as
the minimum between these two lengths: 
\begin{equation}
\label{amp}
\alpha(d_{i,j}):=\min \{j-i, \gamma-j+i\}.
\end{equation}
It is clear that $\alpha(d_{i,j})$ does not depend on the choice of the  labeling.
 \end{nota}

\begin{lemmadefi}
\label{trivia}
Let $\Gamma$ be a $p$-hamiltonian graph with a distinguished hamiltonian cycle.
Set $\gamma:=|V(\Gamma)|  =(2b_1(\Gamma)-2)/({p-2}) $.
\begin{enumerate}
\item
For any    chord $d_{i,j}$ we have
 $1\leq \alpha(d_{i,j})\leq \gamma/2$.
If  $\alpha(d_{i,j})\leq \gamma/2-1$ we say that $d_{i,j}$ is   
 {\emph{short}}.
\item
 Let $d_{i,j}$ be a short chord.
The side of $d_{i,j}$ having length $\alpha(d_{i,j})$ will be called the {\emph{short side}}
of $d_{i,j}$.
\item
If $\alpha(d_{i,j})=\lfloor \gamma/2\rfloor$ for every   chord,
or equivalently, if $\Gamma$ has no short chords, then
  $\Gamma$ is uniquely determined,
  it will be  denoted by $\Pi_{\gamma}^p$ and will be called  the {\emph{$p$-polygon}}
  with $\gamma$ vertices
  (see Figures 3 and 4).
  
  If $\gamma$ is even the graph $\Pi_{\gamma}^p$ has
   $p-2$ chords between $v_i$ and $v_{i+\gamma/2}$  for every  
  $i=1,\ldots,\gamma/2$,   and no other  chord.

  If $\gamma$ is odd then $p$ is even.  
  For every  
  $i=1,\ldots,(\gamma-1)/2$,  the graph $\Pi_{\gamma}^p$
  has $(p-2)/2$ chords between $v_i$ and $v_{i+(\gamma-1)/2}$,  
  $(p-2)/2$ chords between $v_i$ and $v_{i+(\gamma+1)/2}$,
  and no other  chord.

\end{enumerate}
 \end{lemmadefi}
 \begin{proof} 
Since $\Gamma$ has no loops we have,  for any    chord $d_{i,j}$,
  $1\leq \alpha(d_{i,j})$. If $\gamma$ is even (respectively, odd) the maximal amplitude of a chord is obviously $\gamma/2$ (respectively, $(\gamma-1)/2$ ). 
  
  Now let $\gamma$ be even. If there are no short chords, every chord is of type
  $d_{i, i+\gamma/2}$ for $i=1,\ldots ,\gamma/2$. Moreover, every pair of vertices $v_i$, $v_{i+\gamma/2}$
  is joined by exactly $p-2$ chords, because $\Gamma$ is $p$-regular.
  This shows that $\Gamma$ is uniquely determined.

Now suppose that   $\gamma$ is odd, and that $\Gamma$ has no short chord.
Then every chord is either of type  $d_{i, i+(\gamma-1)/2}$  or of type  $d_{i, i+(\gamma+1)/2}$.
Since $|E(\Gamma)|=p\gamma/2$ we have that  $p$ is even;   set $r=(p-2)/2$. For every vertex there are $2r$ chords touching it.

We claim that there are exactly $r$ chords of type $d_{i, i+(\gamma-1)/2}$
and  $r$ chords of type $d_{i, i+(\gamma+1)/2}$ for every $i=1,\ldots, (\gamma-1)/2$.
By contradiction, suppose (with no loss of generality) that there are more than $r$ chords joining $v_1$ with $v_{(\gamma+1)/2}$; hence there are less than $r$ chords joining $v_1$ with $v_{(\gamma+3)/2}$.
But then there are  more than $r$ chords joining $v_2$ with $v_{(\gamma+3)/2}$
and less than $r$ chords joining $v_2$ with $v_{(\gamma+5)/2}$. Continuing in this way we get that
there are less than $r$ chords joining $v_{(\gamma-1)/2}$ with $v_{\gamma}$.
The remaining chords touching $v_{\gamma}$ are the ones touching also $v_{(\gamma+1)/2}$;
since there are already more than $r$ chords of type $d_{1, (\gamma+1)/2}$,
there can only be less than $r$ chords of type $d_{\gamma, (\gamma+1)/2}$. We conclude that there are less than $2r$ chords touching $v_\gamma$. A contradiction.
This shows that $\Gamma$ is uniquely determined.  
    \end{proof}
 \begin{example}
 If $p=3$ we have $\gamma =2b_1(\Gamma)-2$ and $\Pi_{\gamma}^3$ has no multiple edge.
 \begin{figure}[!htp]
\label{Pi34}
$$\xymatrix@=1pc{
&&*{\bullet}\ar@{-}[dd]_>{v_3}^<{v_1} \ar@{-}[dl]_>{v_4} \ar@{-}[dr]^>{v_2}  &&&&
&&*{\bullet} \ar@{-}@/_/[ddr]\ar@{-}[dl]_>{v_6} &*{\bullet}\ar@{-}@/^/[ddl]\ar@{-}[l]_>{v_1}  \ar@{-}[dr]^>{v_3}^<{v_2}
\\
\Pi_4^3=&*{\bullet} \ar@{-}[rr] |!{[r]}\hole & &*{\bullet} &&&
\Pi_6^4=&*{\bullet} \ar@{-}@/_/[rrr]\ar@{-}[rrr]  & &&*{\bullet} 
 \\
&&*{\bullet}\ar@{-}[ru]\ar@{-}[lu]&&&&
&&*{\bullet}\ar@{-}[ruu]\ar@{-}[lu]^<{v_5}&*{\bullet}\ar@{-}[l]  \ar@{-}[ur] _<{v_4}\ar@{-}[luu]
}$$
 \caption{Some $p$-polygons with even number of vertices.}
\end{figure}
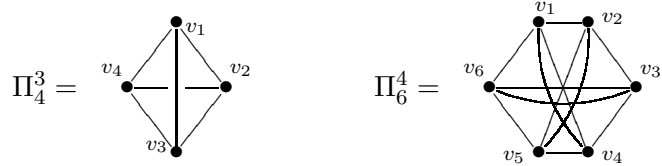

If $\gamma$ is odd, then $\Pi_{\gamma}^p$ has no multiple edges if and only if $p\leq 4$.
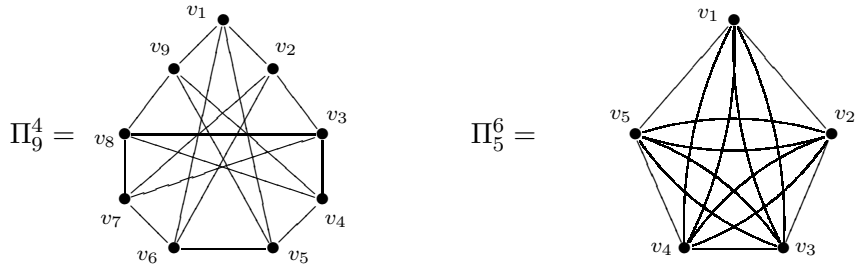
\begin{figure}[!htp]
$$\xymatrix@=1pc{
&&&*{\bullet} \ar@{-}[ddddl] \ar@{-}[ddddr]\ar@{-}@[dddl]&&&&&&&&&
*{\bullet}\ar@{-}[ddll]_>{v_5}_<{v_1} \ar@{-}@/^/[ddddl] \ar@{-}@/_/[ddddl] \ar@{-}@/_/[ddddr]\ar@{-}@/_/[ddddl]\ar@{-}@/^/[ddddr]\ar@{-}[ddrr]^>{v_2}\\
 &&*{\bullet}\ar@{-}[ur]^>{v_1}^<{v_9} \ar@{-}[ddrrr] \ar@{-}[dl]& &*{\bullet}\ar@{-}[ddlll]\ar@{-}[dddll]  \ar@{-}[ul]_<{v_2}\ar@{-}[dr] ^>{v_3}\\
\Pi_9^4  =&*{\bullet}\ar@{-}[drrrr] \ar@{-}[rrrr]  & &&&*{\bullet} \ar@{-}[dllll] &&&
\Pi_5^6  =&&*{\bullet}\ar@{-}@/^/[ddrrr] \ar@{-}@/_/[ddrrr] \ar@{-}@/^/[rrrr] \ar@{-}@/_/[rrrr] & &&&*{\bullet} \ar@{-}@/_/[ddlll] 
\ar@{-}@/^/[ddlll]\\
 &*{\bullet}\ar@{-}[u] ^>{v_8} \ar@{-}[dr] _<{v_7}_>{v_6}&&& &*{\bullet}\ar@{-}[u]\ar@{-}[dl]^<{v_4}^>{v_5} \\
&&*{\bullet}&&*{\bullet}\ar@{-}[ll] \ar@{-}[uuull] &&&&&&&
*{\bullet}\ar@{-}[luu]^<{v_4}&&*{\bullet}\ar@{-}[ll] \ar@{-}[uur] _<{v_3} 
}$$
\caption{Some $p$-polygons with an odd number of vertices.}
\end{figure}
\end{example}
We need a criterion for 3-edge-connectivity.
\begin{lemma}
\label{crit}
\begin{enumerate}
\item
\label{crit1}
Let $\Gamma$ be a graph such that for every edge $e$ there exist two distinct  cycles
$\Delta_1$ and $\Delta_2$ in $\Gamma$ and such that $E(\Delta_1)\cap E(\Delta_2)=\{e\}$. Then $\Gamma$ is 3-edge-connected.
\item
\label{crit2}
Let $\Gamma_1$ be a 3-edge-connected graph  and let $\Gamma_2$ be a graph strongly linked to $\Gamma_1$,
so that $\Gamma_1/e_1=\Gamma_2/e_2$ with $e_i\in E(\Gamma_i)$ (notation in Def.~\ref{link}).
Then $\Gamma_2$ is 3-edge-connected if it contains   two cycles
$\Delta_1\neq \Delta_2$   such that $E(\Delta_1)\cap E(\Delta_2)=\{e_2\}$.
\item
\label{crit3}
The $p$-polygon $\Pi_{\gamma}^p$ is 3-edge-connected for every $p\geq 3$.
\end{enumerate}
\end{lemma}
\begin{proof}
For   part (\ref{crit1}), we notice that $\Gamma$ has no separating edges
(a separating edge is not contained in any cycle).
Suppose by contradiction that $\Gamma$ is not 3-edge-connected;
let $(e,e')$ be a separating pair of edges of $\Gamma$. By \cite[Lemma 2.3.2 (iv) and (iii)]{CV},
$(e,e')$ is a separating pair if and only if $e$ and $e'$ belong to the same cycles of $\Gamma$.
By our assumption, this is clearly impossible.

Now part (\ref{crit2}). The graph $\Gamma_1/e_1=\Gamma_2/e_2$ is 3-edge-connected as $\Gamma_1$ is.
Therefore any separating pair of edges of $\Gamma_2$ must contain $e_2$. The proof of part (\ref{crit1}) shows that
our hypothesis implies that $e_2$ is not contained in any separating pair of edges, hence we are done.

To prove part (\ref{crit3}) we use again part (\ref{crit1}). 
Pick a chord
 $d_{i,j}$; then there obviously exist two cycles having only $d_{i,j}$  as common edge:
 just take the two cycles obtained by adding to  $d_{i,j}$ one of its two sides
 (terminology in subsection~\ref{norm}).
 To prove that we can apply (\ref{crit1}) on the remaining edges we need to distinguish two cases, according to the parity of $\gamma$.
 
Suppose $\gamma$ even.
By Lemma~\ref{trivia} in  $\Pi_{\gamma}^p$
 there exists at least one chord  $d_{i,i+\gamma/2} $ joining $v_i$ with $v_{i+\gamma/2}$, 
for every $i=1,\ldots, \gamma/2$.   
Pick an edge which is not a chord, $e=e_1$.
Now $\Pi_{\gamma}^p$ contains the chords $d_{1,\gamma/2+1}$ and $d_{2,\gamma/2+2}$.
Then
$\Delta_1=(e_1,\ldots, e_{\gamma/2}, d_{1,\gamma/2+1})$ and 
$\Delta_2=(e_1, e_{\gamma}, \ldots, e_{\gamma/2+2},d_{2,\gamma/2+2})$
are two cycles having only $e$  as common edge.
Therefore  $\Pi_{\gamma}^p$ is 3-edge-connected.

Now suppose that $\gamma$ is odd; again we use  (\ref{crit1}). 
By Lemma~\ref{trivia} in  $\Pi_{\gamma}^p$
 there exists at least one chord    joining $v_i$ with $v_{i+(\gamma-1)/2}$, 
 and at least one chord joining $v_i$ with $v_{i+(\gamma+1)/2}$.   
Let $e=e_1$ be  an edge which is not a chord. Let
$\Delta_1=(e_1, d_{2,(\gamma+3)/2}, d_{1,(\gamma+3)/2})$ and 
$\Delta_2=(e_1, e_2,\ldots, e_{(\gamma-1)/2},d_{1,(\gamma+1)/2})$; these are two cycles whose only edge in common is $e_1$.
Hence $\Pi_{\gamma}^p$ is 3-edge-connected.
\end{proof}
  We say that   two chords $d_{i,j}$ and $d_{k,l}$  do not cross  if
 $i<j<k<l$.
\begin{lemma}
\label{short}
Let $\Gamma$ be a $p$-hamiltonian graph with a distinguished hamiltonian cycle(cf. \ref{norm}).
  Let $d_{i,j}$ be a short chord.
Then there exists a short chord $d_{k,l}$   with $j<k$ (i.e. $d_{i,j}$ and $d_{k,l}$  do not cross)
and such that the short side of $d_{i,j}$
does not intersect the short side of $d_{k,l}$.
 \end{lemma}
\begin{proof}
We denote by $\Delta$ the fixed hamiltonian cycle.
We may assume that $i=1$, so that the given chord $d_{1,j}$ has $j\leq \gamma/2$
 (i.e. the short side of $d_{1,j}$ has vertices $v_1, v_2,\ldots, v_j$).
 We must prove that there exists a short chord $d_{k,l}$
 such that
 
 (a) $j<k$  ($d_{1,j}$ and  $d_{k,l}$ do not  cross).
 
 (b) $l-k< \lfloor\gamma/2\rfloor$ (the short side of $d_{k,l}$ has vertices $v_k, v_{k+1},\ldots, v_{l-1},v_l$).
 
 Let us denote by $D$ the set of chords satisfying (a);
 we begin by bounding $|D|$ from below.
Consider the $j$ vertices   $v_1, v_2,\ldots, v_j$; 
 there are at most $p-2$ chords touching  each of them. Therefore
the total number of distinct chords touching these vertices is at most
  $j(p-2)-1$  (to explain the ``$-1$" notice that the chord $d_{1,j}$ joins $v_1$ with $v_j$, hence it 
  must not be counted twice).
Since $\Gamma$ has  $b-1$ chords, we get
\begin{equation}
\label{D}
|D|\geq b-1 - j(p-2)+1=b-  j(p-2).
\end{equation}
 In particular, since $b=1+(p-2)\gamma/2$ and 
 $j\leq \gamma/2$, we have that $|D|\geq 1$, i.e.
  $D$ is not empty. To prove that there exists at least one chord in $D$ satisfying (b) we argue by contradiction.
Suppose that every chord    $d_{k,l} \in D$ satisfies $ l-k\geq  \gamma/2$. This is to say that the path
$\Lambda\subset \Delta$ 
from $v_{j+1}$ to $v_{\gamma}$  
contains a side of length at least $\gamma/2$ for every non multiple chord in $D$.
Let us restrict our attention to the subgraph 
$ 
\Gamma '=\Gamma \smallsetminus \{v_1,\ldots, v_j\},
$  
obtained   by removing the vertices $\{v_1,\ldots, v_j\}$ and all edges adjacent to them.
So, $\Gamma '$ is made of $\Lambda$ together with every chord in $D$.
Now,   two vertices of $\Lambda$ are joined  by a chord of $D$ only if
they are separated by at least $\gamma/2$ edges. 
Moreover, every two vertices can be joined by at most $p-2$ chords.
Therefore   the length of the path $\Lambda$
satisfies, using  \eqref{D},
$$
\text{length}(\Lambda)\geq 
\frac{\gamma}{2}+\frac{|D|}{p-2} -1 \geq  
\frac{b-1}{p-2}+\frac{b}{p-2} -j-1=
\frac{2b-1}{p-2}  -j-1>\gamma-j-1
$$
(since $\gamma=\frac{2b-2}{p-2}$). On the other hand $\Lambda$ is a path 
from $v_{j+1}$ to $v_{\gamma}$, whose 
  length is easily computed:
  $$
\text{length }(\Lambda)=\gamma-(j+1)=\gamma-j-1,
$$
which is in contradiction with the above estimate on $\text{length}(\Lambda)$.
\end{proof}

\subsection{Proof of the linkage theorem}
\label{linksec}
\begin{nota}
{\it Twisting pairs of chords in a $p$-hamiltonian graph.}
Let $\Gamma$ be a  $p$-hamiltonian graph with a distinguished hamiltonian cycle,
as in \ref{norm}; pick  two chords $d_{i,j}$ and $d_{k,l}$. 
In this subsection we momentarily suspend the general convention $i<j$ and $k<l$ (which would be too restrictive).
We introduce the graph $\Gamma'$ obtained from $\Gamma$ by 
swapping two endpoints of the above chords. So, $\Gamma'$ is  obtained from $\Gamma$ by
replacing 
the chord $d_{i,j}$ with a new chord, $d_{i,k}$, joining $v_i$ and $v_k$,
and by
replacing 
  $d_{k,l}$ with a chord  $d_{j,l}$. 
  Everything else is left unchanged.
  We shall say that $\Gamma '$ is a {\it twist} of $\Gamma$, and that 
 $\Gamma'$ is obtained from $\Gamma$ 
by {\it twisting} the pair of chords $(d_{i,j},d_{k,l})$ into the pair $(d_{i,k},d_{j,l})$.
We shall also say that we swapped the end points $v_j$ and $v_l$.

With no loss of generality we may set  $i=1$; the graph $\Gamma '$ is obviously a  $p$-regular hamiltonian graph;
a distinguished hamiltonian cycle will be   naturally induced by the one  of $\Gamma$.
So the vertices of $\Gamma$ and $\Gamma '$ will have the same names,
 and all the edges of $\Gamma$ other than $d_{1,j}$ and $d_{k,l}$
 correspond to edges of $\Gamma'$ other than $d_{1,k}$ and $d_{j,l}$.

The picture below represents two  3-hamiltonian graphs  related by twisting a pair of chords
(the  dotted chords are the ones that are not changed).

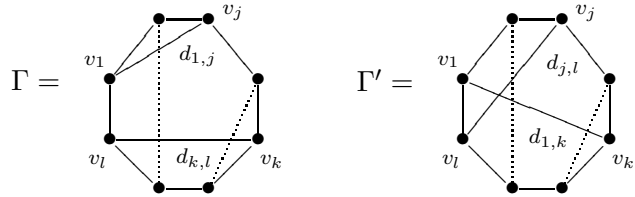
\begin{figure}[!htp]
\label{twist}
$$\xymatrix@=1pc{
&&*{\bullet}\ar@{..}[ddd]\ar@{-}[dl]_>{v_1} &*{\bullet}\ar@{-}[l]\ar@{-}[dr]^<{v_j}&&& 
&&*{\bullet}\ar@{..}[ddd]\ar@{-}[dl]_>{v_1} &*{\bullet}\ar@{-}[l]\ar@{-}[dr]^<{v_j}
\\
\Gamma= &*{\bullet} \ar@{-}[d]  \ar@{-}[urr] _<(.7){d_{1,j}}& &&*{\bullet} \ar@{-}[d]&& \Gamma'= 
&*{\bullet} \ar@{-}[d] & &&*{\bullet} \ar@{-}[d]
 \\
&*{\bullet}  && &*{\bullet} \ar@{-}[lll]^<(.4){d_{k,l}}&& 
  &*{\bullet}  \ar@{-}[uurr]_>(.8){d_{j,l}}  && &*{\bullet} \ar@{-}[lllu]^<(.3){d_{1,k}}
\\
&&*{\bullet}\ar@{-}[lu]^>{v_l}&*{\bullet}\ar@{-}[l]  \ar@{-}[ur]_>{v_k}\ar@{..}[uur]&&& 
&&*{\bullet}\ar@{-}[lu]^>{v_l}&*{\bullet}\ar@{-}[l]  \ar@{-}[ur]_>{v_k}\ar@{..}[uur]
}$$

\caption{Twisting  $ d_{1,j}$ and $d_{k,l}$ into  $ d_{1,k}$ and $d_{j,l}$.}
\end{figure}
\end{nota}
The following technical lemma is used in the proof of Theorem~\ref{main}.
\begin{lemma}
\label{linktwist}
Let $\Gamma$ be a $p$-hamiltonian graph with a distinguished hamiltonian cycle.
\begin{enumerate}
\item
If $\Gamma '$ is a twist of $\Gamma$, then $\Gamma$ and $\Gamma '$ are linked.
\item
Let $\Gamma$ be 3-edge-connected   and fix a chord $d_{i,j}$ of $\Gamma$, with $i<j$.
Let $d_{j+1,*}$ be a chord of $\Gamma$ starting at the vertex $v_{j+1}$; suppose that either (a) or (b) below
hold.
\begin{enumerate}[(a)]
\item
$d_{j+1,*}=d_{j+1,h}$ with $j+1<h$ (i.e. $d_{j+1,*}$ does not cross $d_{i,j}$).
\item
$d_{j+1,*}=d_{h,j+1}$ with $i<h<j$ and
there exists a third chord $d_{x,y}$ such that
$1\leq i<h<x<j<j+1<y$.
\end{enumerate}
Then the graph  obtained by twisting ($d_{i,j},d_{j+1,*}$) into
 ($d_{i,j+1},d_{j,*}$) is 3-edge-connected and strongly linked to $\Gamma$.
\end{enumerate}
  \end{lemma}
\begin{proof}
We can assume $i=1$ so that $d_{i,j}=d_{1,j}$. The edges of the distinguished hamiltonian cycle $\Delta$  will be called, as usual, $e_1, e_2,\ldots, e_{2b-2}$
with $e_i$ joining $v_i$ with $v_{i+1}$.
Let $\Gamma'$ be a twist of $\Gamma$.
We prove $\Gamma'$  is linked to $\Gamma$
by induction on $k-j$
(i.e. on the distance along   $\Delta$ of the two   swapped vertices). If $k=j+1$ let $e$ be the edge of $\Delta$ between $v_j$ and $v_{j+1}$. Then the graph obtained from $\Gamma$ by contracting $e$ is the same as the graph
obtained from $\Gamma'$ by contracting $e$; hence $\Gamma$ and $\Gamma '$ are strongly linked.
Now assume $k-j\geq 2$. Let $\Gamma_1$ be the graph obtained from $\Gamma$ by twisting the chord $d_{1,j}$ with a chord ending at $v_{j+1}$, denoted $d_{j+1,*}$.
So, in $\Gamma_1$ we have the chords $d_{1,j+1}^1$ and $d_{j,*}^1$, where the superscript  keeps track of the graph   to which the chords belong. We already proved that $\Gamma$ and $\Gamma_1$ are linked. Now consider the graph $\Gamma_2$ obtained from 
$\Gamma_1$ by twisting $d_{1,j+1}^1$ and $d_{k,l}^1$, replacing them with
$d_{1,k}^2$ and $d_{j+1,l}^2$; since $k-(j+1)<k-j$, by induction $\Gamma_1$ and $\Gamma_2$ are 
linked. Finally, let $\Gamma_3$ be  
obtained from 
$\Gamma_2$ by twisting $d_{j+1,l}^2$ and $d_{j,*}^2$, replacing them with
$d_{j,l}^3$ and $d_{j+1,*}^3$. Again by induction
(we are swapping $v_j$ and $v_{j+1}$) $\Gamma_3$ is linked to $\Gamma _2$, and therefore $\Gamma_3$ is linked to $\Gamma$. It is obvious that $\Gamma_3=\Gamma'$.

Let us prove the second part.
Consider $e_j$, the edge between $v_j$ and $v_{j+1}$
(which, abusing notation as usual, is an edge of both $\Gamma$ and $\Gamma'$).
It is easy to check  that $\Gamma $ and $\Gamma'$ are strongly linked,
as the graphs $\overline{\Gamma}$ and  $\overline{\Gamma'}$
obtained from  $\Gamma $ and $\Gamma'$ by contracting $e_j$ are obviously isomorphic.
We need to prove that if either (a) or (b) holds, then $\Gamma '$ is  3-edge-connected if $\Gamma$ is.
By Lemma~\ref{crit} (\ref{crit2}), it is enough to show that
the edge $e_j$   belongs to two distinct cycles $\Delta_1$ and $\Delta _2$ of $\Gamma '$, such that $E(\Delta_1)\cap E(\Delta_2)=\{e_j\}.$

Suppose (a) holds, so 
we are twisting $(d_{1,j},d_{j+1,h})$ into $(d_{1,j+1},d_{j,h})$, with $h> j+1$.
Then in $\Gamma '$ we have the cycles $\Delta_1$ and $\Delta_2$ whose edge sets are
$$
E(\Delta_1)=\{ e_j,d_{1,j+1},e_1,\ldots, e_{j-1}  \}
$$
and 
$$
E(\Delta_2)=\{e_j, e_{j+1},\ldots,e_{h-1}, d_{j,h} \}.
$$
 It is clear 
that $\Delta_1$ and $\Delta_2$ are cycles and that $E(\Delta_1)\cap E(\Delta_2)=\{e_j\}$.

Now assume (b). 
We are twisting $(d_{1,j},d_{h,j+1})$ into $(d_{1,j+1},d_{h,j})$.
 Let $d_{x,y}$ be a chord crossing both $d_{h,j}$ and $d_{1,j}$.
We   have
$$
1<h<x<j <j+1<y.
$$
Now the edges of the two cycles containing $e_j$ and sharing no other edge  are
$$
E(\Delta_1)=\{ e_j,d_{1,j+1},e_1,\ldots, e_{h-1}, d_{h,j}  \}
$$
and
$$
E(\Delta_2)=\{ e_j,e_{j+1},\ldots, e_{y-1},d_{x,y},e_{x}, e_{x+1},\ldots, e_{j-1}  \}.
$$
Since $1<h<x$ we have $E(\Delta_1)\cap E(\Delta_2)=\{e_j\}$.
\end{proof}

We are ready to prove the linkage theorem.

\begin{thm}
\label{main}
Let $\Gamma_1$ and $\Gamma _2$ be    $p$-regular  graphs  with
$ b_1(\Gamma_1)=b_1(\Gamma_2)$.

Then $\Gamma_1$ and $\Gamma _2$ are linked.

If $\Gamma_1$ and $\Gamma _2$ are 3-edge-connected, then they are 3-linked.
\end{thm}
\begin{proof}
By Proposition~\ref{poly}   we can assume that $\Gamma_1$ and $\Gamma _2$ are $p$-hamiltonian.

We shall prove the theorem by showing that every $p$-hamiltonian graph  $\Gamma$ is linked to
the $p$-polygon    $\Pi_\gamma^p$, where  $\gamma=\frac{2b-2}{p-2}$ and $b=b_1(\Gamma)$.
Moreover, if $\Gamma $ is 3-edge-connected, we will prove that it is 3-linked to $\Pi_\gamma^p$,
which is 3-edge-connected by Lemma~\ref{crit}.

Let us fix a distinguished hamiltonian cycle $\Delta$  of $\Gamma$   and use the notation of \ref{norm}.  Now 
set
$$
\epsilon(\Gamma):=\sum  \Bigl(\lfloor \gamma/2\rfloor-\alpha(d_{i,j}) \Bigr)
$$
where the sum is over all the   chords of $\Gamma$.
By Lemma~\ref{trivia} we have $\epsilon(\Gamma)\geq 0$, and 
 $\epsilon(\Gamma)=0$ if and only if  $\Gamma$ has no short chord, if and only if
$\Gamma=\Pi_\gamma^p$.

We will prove the theorem by induction on $\epsilon(\Gamma)$. By what we just observed,
if $\epsilon(\Gamma)=0$ there is nothing to prove, so the induction basis is settled.

Assume now  that $\epsilon(\Gamma)>0$ and let us pick a short chord;
  we may call it $d_{1,j}$  and   assume that $j\leq \gamma/2$.
By  Lemma~\ref{short} we have that
there exist chords $d_{k,l}$
satisfying  
\begin{equation}
\label{kl}
1<j<k<l\  \    \  \  \text{ and }\  \   \  \   l-k<\lfloor  \gamma/2\rfloor.\  \  
\end{equation}
We can assume (up to  changing the labeling of the vertices) that there exists one of them such that 
the path (in $\Delta$) from $v_j$ to $v_k$ is not longer than the path from $v_l$ to $v_1$; i.e. we can assume that
\begin{equation}
 \label{shift} 
k-j\leq \gamma+1-l.
\end{equation}
We shall pick the pair $(d_{1,j}, d_{k,l})$ such that $k-j$,
i.e. the length of the path in $\Delta$ from $v_j$ to $v_k$,
is   minimal  with respect to all pairs satisfying  \eqref{kl} and \eqref{shift};
we shall refer to this as the ``minimality property" of $(d_{1,j}, d_{k,l})$. 
 
 Now that we have fixed our two chords, we can assume, up to switching them and changing the labeling on the vertices,
 that 
 \begin{equation}\label{mini}
j-1= \alpha(d_{1,j})\leq \alpha(d_{k,l})=l-k.
 \end{equation}
 With these settings, we have
  \begin{equation}
  \label{kb}
 k\leq   \gamma/2 +1,\    {\text{  more exactly }}\   k\leq\lfloor  \gamma/2\rfloor+1.
 \end{equation}
Let us prove (\ref{kb}) by contradiction. Suppose  $k\geq\lfloor  \gamma/2\rfloor+2$. Then    (\ref{mini}) implies
$$
\lfloor  \gamma/2\rfloor+2\leq k\leq l-j+1.
$$
Now, by (\ref{shift}), we have
 $ 
 l-j+1\leq \gamma+2-k. 
 $ 
 Therefore 
 $ 
\lfloor  \gamma/2\rfloor+2\leq \gamma+2-k$ and hence $ k\leq \lfloor  \gamma/2\rfloor+1$;
 a contradiction.
 
\begin{claim}
\label{claim}
Let $\Gamma '$ be the graph obtained from $\Gamma$ by
twisting the pair of chords $(d_{1,j}, d_{k,l})$ into the pair 
 $(d_{1,k}, d_{j,l})$.
Then  $\epsilon (\Gamma ')<\epsilon (\Gamma)$.
 \end{claim}
 To prove  the claim,
consider a chord $d'$ of $\Gamma '$.
For notational clarity, we will denote by $d'_{*,*}$ the chords of $\Gamma '$.
If $d'$ is not equal to $d'_{1,k}$ or $d'_{j,l}$, then $d'$ corresponds to a unique chord $d$ of $\Gamma $ such that $\alpha(d)=\alpha(d')$.
Therefore we have
\begin{equation}
\label{diff}
\epsilon(\Gamma)-\epsilon(\Gamma')=-\alpha(d_{1,j})-\alpha(d_{k,l})+\alpha(d'_{1,k})+\alpha(d'_{j,l}).
\end{equation}
We know that
$\alpha(d_{1,j})= j-1$ and $\alpha(d_{k,l})=l-k$, by construction and by  \eqref{kl}.
Furthermore, by (\ref{kb}) we have $\alpha(d'_{1,k})= k-1$ .

To compute the remaining term we need to distinguish two cases.

\noindent
  Case 1:  $l-j\leq\gamma/2$.
Then
$\alpha(d'_{j,l})=l-j$.
Therefore 
$$
\epsilon(\Gamma)-\epsilon(\Gamma')=1-j+k-l+k-1+l-j=2k-2j\geq 2
$$
by  \eqref{kl}. So the claim is proved in this case.

\noindent
Case 2:  $l-j\geq \gamma/2+1$.
Now
 $\alpha(d'_{j,l})=\gamma+j-l$.
 Therefore
$$
\epsilon(\Gamma)-\epsilon(\Gamma')=1-j+k-l+k-1+\gamma+j-l=\gamma+2(k- l)\geq 2
$$
as $l-k<\lfloor  \gamma/2\rfloor$ by (\ref{kl}). The claim is proved.

\

 Lemma~\ref{linktwist} says that $\Gamma$ and $\Gamma'$ are linked.
By the claim we may apply induction, getting that  $\Gamma'$ is linked to $\Pi_{\gamma}^p$; hence
 the first part of the Theorem is proved.

 Before continuing,
 we analyze the chords having one end at a vertex $v_g$, with
 $j+1\leq g\leq k-1$. Let $d_{g,*}$ be one such chord.
 We claim that with our choice of the pair $(d_{1,j}, d_{k,l})$, we have
\begin{equation}
\label{midchord}
d_{g,*}=d_{g,m},\  \  \   m\geq k.
\end{equation}
By contradiction, suppose $m< k$.
If $m< g$ we have (as $g\leq k-1$ and $m\geq 1$)
$$
g-m\leq k-1-1\leq \gamma/2-1
$$ 
by (\ref{kb}).
Therefore $d_{m,g}$ satisfies the properties satisfied by $d_{1,j}$: it is a short chord
whose short side does not intersect the short side  of $d_{k,l}$,
and it verifies (\ref{shift}),
i.e. the path from $v_g$ to $v_k$  is not shorter than the path from
$v_l$ to $v_m$. Now,
 the path from $v_g$ to $v_k$ is obviously shorter than the path from $v_j$ to 
$v_k$, contradicting the minimality property of $(d_{1,j}, d_{k,l})$. 

Suppose now that $g< m< k$.
Again, $d_{g,m}$ satisfies  (\ref{shift}) and $v_m$ is closer to $v_k$ than $v_j$. Therefore,
in order  to respect the minimality property of $(d_{1,j}, d_{k,l})$, we must have
$m-g\geq \gamma/2$. This  implies ($m\leq k-1\leq \gamma/2$ by (\ref{kb}) and $g\geq j+1$)
$$
\gamma/2\leq m-g \leq \gamma/2 -j-1
$$
 which is obviously impossible. \eqref{midchord} is proved.

To finish the proof of  the theorem,
it is enough to  show that, if $\Gamma$ is 3-edge-connected, then $\Gamma '$ is 3-edge-connected  and
 3-linked to $\Gamma $.
 To do that  we shall factor  the twist of $(d_{1,j}, d_{k,l})$ into  $(d_{1,k}, d_{j,l})$
by
 a series of twists swapping consecutive vertices, each of which preserves   3-edge-connectivity.
 We do that with two sets of twists. 
To define   the first set, we make a choice of a  chord   $d_{h+1,*}$ for every $j\leq h\leq k-1$.
 This choice will be irrelevant.

\

{\it (I.1)} \  \  \  \  \  Twist  $(d_{1,j}, d_{j+1,*})$ \  into  $(d_{1,j+1}, d_{j,*}).$

{\it (I.2)}  \  \  \  \  \  Twist  $(d_{1,j+1}, d_{j+2,*})$ into  $(d_{1,j+2}, d_{j+1,*}).$

.........

{\it  (I.h+1-j)} Twist  $(d_{1,h}, d_{h+1,*})$ into  $(d_{1,h+1}, d_{h,*}),$ with $j\leq h\leq k-1$.
 
.........

{\it  (I.k-j)}  \  \  \  \ Twist  $(d_{1,k-1}, d_{k,l})$ into  $(d_{1,k}, d_{k-1,l})$
 
 \
 
Observe that in each of the above twists,  
  the two chords getting twisted, $d_{1, h}$ and $d_{h+1,*}$, do not cross,
  i.e. $d_{h+1,*}=d_{h+1,m}$ with $m>h+1$.
  This is obvious for the last step, $(I.k-j)$, as $1<k-1<k<l$.
For the remaining steps, for which $h\leq k-2$, we use
 \eqref{midchord}, according to which every $d_{h+1,*}$  is of type $d_{h+1,m}$
with $m\geq k$. Hence $1<h<h+1\leq k-1<m$, as claimed.
  
Therefore condition (a) of  Lemma~\ref{linktwist} holds, and we 
 conclude that the graph  $\Gamma''$, obtained after the above set of twists, is 3-edge-connected
 and 3-linked to $\Gamma$.
 
 Notice that $\Gamma''$ contains  the chord $d_{1,k}$  and the chord
 $d_{k-1,l}$.
The second set of twists, starting from $\Gamma ''$ is the following.

{\it (II.1)}  \  \  \  \    Twist  $(d_{k-1,l}, d_{k-2,*})$ into  $(d_{k-2,l}, d_{k-1,*})$.

{\it (II.2)}  \  \  \   \  Twist  $(d_{k-2,l}, d_{k-3,*})$ into  $(d_{k-3,l}, d_{k-2,*})$.

.........
 
 {\it (II.k-h)}  \  \  Twist  $(d_{h,l}, d_{h-1,*})$ into  $(d_{h-1,l}, d_{h,*})$, where $j+1\leq h\leq k-1$.

.........

{\it (II.k-j-1)} Twist  $(d_{j+1,l}, d_{j,*})$ into  $(d_{j,l}, d_{j+1,*})$ 
 
 \
 
 where the chords $d_{h-1,*}$ are those chosen for   the first set of twists.
  Observe that the chord $d_{1,k}$ 
 (which lies in every graph appearing in   the above twists)
 crosses every chord $d_{h,l}$  with  $j+1\leq h\leq k-1$.
 If $d_{1,k}$  crosses also $d_{h-1, *}$ Lemma~\ref{linktwist} applies
 to the step (II.k-h) above (condition (b) of Lemma~\ref{linktwist} holds), and hence 3-edge-connectivity is preserved
 (to fit in precisely with the notation of Lemma~\ref{linktwist}, one translates the starting vertex after $v_h$, 
 sets $h-1=j$ and $h=j+1$ so that $d_{h-1,*}$ becomes $d_{i,j}$ and $d_{h,l}$ becomes $d_{j+1,*}$.).

What if 
$d_{1,k}$  does not cross   $d_{h-1, *}$?
Recall that by \eqref{midchord} we have $d_{h-1, *}=d_{h-1,m}$ with $m\geq k$.
Therefore  $d_{1,k}$ does not cross $d_{h-1,m}$ only if  $k=m$.
Let us show that twisting $(d_{h,l},d_{h-1,k})$ 
into $(d_{h-1,l},d_{h,k})$ 
preserves 3-edge-connectivity; let 
$\widetilde{\Gamma}$ be the graph obtained after this twist.
By Lemma~\ref{crit}(\ref{crit2})
it suffices to show that $\widetilde{\Gamma}$ contains
  two cycles, $\Delta_1$ and $\Delta_2$,   whose only edge in common is
  the edge $e_{h-1}$ (joining the two swapped vertices $v_{h-1}$ and  $v_h$).
Here are the two cycles
  $$
  \Delta_1=(e_{h-1}, d_{h,k},d_{1,k}, e_1,\ldots, e_{h-2})
  $$ 
 and
 $$
  \Delta_2=(e_{h-1},  e_h,\ldots, e_{l-1},d_{h-1,l}).
  $$ 
 
 Therefore the graph $\Gamma'''$ obtained from $\Gamma$
by  our two sets of twists  is 3-edge-connected and 3-linked to $\Gamma$.
Let us check that  $\Gamma'''$ coincides with  the $\Gamma '$ of   Claim~\ref{claim}.
The chords $d_{1,j}$ and  $d_{k,l}$ of $\Gamma$ are twisted into  $d_{1,k}$ and  $d_{j,l}$
in $\Gamma'$ and $\Gamma'''$.
The remaining chords of $\Gamma$ and $\Gamma'$ are the same.
The chord
 $d_{h+1,m}\in E(\Gamma)$   with $j\leq h\leq k-2$, in the first set of twists,
 is changed into the chord $d_{h,m}\in E(\Gamma'')$, which is changed  back into $d_{h+1,m}\in E(\Gamma''')$
by the second set of twists. 
 All other chords of $\Gamma$ are not touched by our twists.
So $\Gamma'''=\Gamma '$ and we are done.
 \end{proof}

\section{Moduli of Tropical curves}
\label{tropsec}

\subsection{Tropical curves and tropical equivalence}
\label{list}
In this subsection we recall   several basic facts. 
The original definition of a tropical curve   can be given in terms of metric graphs,
by \cite{MIK3} or  \cite{MZ}. In the following, we use   a  terminology slightly different from the cited references.
Recall that our graphs are assumed connected.
\begin{enumerate}[$\bullet$]
\item
A {\it pure  tropical curve} is a pair $(\Gamma, \ell)$ where $\Gamma$ is a graph 
  and  $\ell$ is a {\it length} function on the edges 
$$ 
\ell:E(\Gamma)\to \R_{> 0}\cup \{\infty\}
$$ 
such that $\ell (e)=+\infty$ if and only if $e$ is   adjacent to a $1$-valent vertex.
The genus of $(\Gamma,\ell )$ is   $g(\Gamma,\ell )=b_1(\Gamma)$.

\item
More generally, following \cite{BMV}, a {\it (weighted) tropical curve} is a triple $(\Gamma, w, \ell)$
where $\Gamma$ is a graph, $w: V(\Gamma)\to \Z _{\geq 0}$ a {\it weight} function on the
vertices,
  and  $\ell$  a  length  function  
$$
\ell:E(\Gamma)\to \R_{> 0}\cup \{\infty\}
$$ 
such that $\ell (e)=+\infty$ if and only if $e$  is adjacent to a $1$-valent vertex of weight $0$.

The genus of $(\Gamma,w,\ell )$ is    defined as follows:
\begin{equation}
\label{gen}
g(\Gamma,w,\ell )=g(\Gamma, w)=b_1(\Gamma)+\sum_{v\in V(\Gamma)}w(v).
\end{equation}

By ``tropical curve" without attribute we shall mean a weighted tropical curve.
If  $w=\mo$, i.e. $w(v)=0$ for every $v\in V(\Gamma)$, the weighted tropical curve is pure.

   \item   Two tropical curves are {\it (tropically) equivalent} if they can be obtained from one  another
   by adding or removing
    2-valent vertices of weight 0,
   or
   1-valent vertices of weight 0, together with their adjacent edge.
   \end{enumerate}

\begin{nota}{\it Pointed tropical curves.}
Before giving more details, we want to extend our discussions to
curves with points on them,  
 so-called ``pointed tropical curves". 
 First, we introduce a generalized  notion of graphs, namely, graphs with legs. 
 Here is the combinatorial definition.
\begin{defi}
\label{gdef}
A    graph   $\Gamma$ with $n$ {\it legs} is the following set of data:
\begin{enumerate}
\item
A finite non-empty set $V (\Gamma)$, the set of {\it vertices}.
  \item
   A finite set $H(\Gamma)$, the set of {\it half-edges}.
  \item 
An involution 
   $ 
\iota: H(\Gamma)\to H(\Gamma)
 $ 
 with  $n$ fixed points
  called  the  {\it legs} of $\Gamma$;   the set of legs is denoted by $L(\Gamma)$.

 A pair $e=\{h ,  \iota(h)\}$ of distinct elements   in $H(\Gamma)$   is called an {\it edge};
  the set of edges is denoted by $E(\Gamma)$.
   \item
A  map $\epsilon:H(\Gamma)\to V(\Gamma)$.
 \end{enumerate}

 If $\epsilon(h)=v$ we say that $h$  is adjacent to $v$, or that $v$ is its endpoint.
The {\it valency} of     $v\in V(\Gamma)$ is the number $|\epsilon^{-1}(v)|$ of half-edges adjacent to $v$.
We say that $\Gamma$ is $p$-regular if every $v\in V(\Gamma)$ has valency $p$.


 \end{defi}
 
It is clear how to associate to the above combinatorial object a topological space. 
Namely
let $\Gamma$ be  a graph as defined above,  with vertex set $V$ and  edge set $E$.
The topological graph associated to it has  $V$ as the set of  $0$-cells; then we add a $1$-cell for every $e=\{ h , \iota(h) \} \in E$,
so that the boundary of this $1$-cell is $\{ \epsilon(h), \epsilon(\iota(h))\}$.
If $\Gamma$ has a non empty set of legs $L$, we add  
an open $1$-cell for every $h\in L$ in such a way that one extreme of the $1$-cell contains $\epsilon(h)$ in its closure.

Of course, if $L$ is empty we have the same graphs   treated in the previous section of the paper.
We shall henceforth view graphs with legs also as topological spaces, and we shall freely switch between the combinatorial and the topological viewpoint. 
As in the previous part of the paper, we shall assume that all our graphs are connected.

Now, a point of a tropical curve can be  efficiently represented by a leg of the corresponding graph.
Here is a list of basic definitions and properties, the first of which generalize those stated in Subsection~\ref{list};   see \cite{CHBK} for details and examples.
 \begin{enumerate}[{\bf(1)}]

\item
An {\it $n$-pointed tropical curve} is a triple $(\Gamma, w, \ell)$
where $\Gamma$ is a graph with $n$ legs, $w: V(\Gamma)\to \Z _{\geq 0}$ a {\it weight} function on the
vertices,
  and  $\ell$ is a {\it length} function  
$$
\ell:E(\Gamma)\cup L(\Gamma)\to \R_{> 0}\cup \{\infty\}
$$ 
such that $\ell (x)=+\infty$ if and only if either
$x\in L(\Gamma)$  or $x$ is an edge   adjacent to a $1$-valent vertex of weight $0$.

The legs of $\Gamma$ are the marked points of the curves.
The genus of $(\Gamma,w,\ell )$ is     $g(\Gamma,w,\ell )=g(\Gamma, w)$ as defined in \eqref{gen}.

\item 
As before, an $n$-pointed    tropical curve is called {\it pure} if $w=\mo$.

An   $n$-pointed    tropical curve is called {\it regular} if it is pure and if
its underlying graph $\Gamma$ is 
$3$-regular.

   \item
 The pair $(\Gamma, w)$   is called a {\it weighted graph} (with $n$ legs);   we say that $(\Gamma, w)$ is the  combinatorial type of the   curve  $(\Gamma, w,\ell)$.
 
   \item $(\Gamma, w)$ is called {\it stable} if every vertex of weight $0$ has valency at least $3$,
   and every vertex of weight $1$ has valency at least $1$
   (in other words, an isolated vertex is stable only if it has weight at least 2; see \cite[Ex. 2.4.7]{CHBK} for more on this point).

   Stable graphs of genus $g$ with $n$ legs exist if and only if $2g-2+n>0$.
   
   \item   Two $n$-pointed tropical curves are {\it (tropically) equivalent} if they can be obtained from one  another
   by adding or removing
   
   (a) 2-valent vertices of weight 0,
   
   or
   
   (b) 1-valent vertices of weight 0, together with their adjacent edge.
   
Tropical equivalence preserves the genus and the number of legs.

\item   Suppose $2g-2+n>0$. Every tropical equivalence class of $n$-pointed tropical curves   contains a unique representative whose combinatorial type is stable.

\item 
\label{iso}Two $n$-pointed tropical curves $(\Gamma_1,w_1,\ell_1)$ and $(\Gamma_2,w_2,\ell_2)$ are
{\it isomorphic} 
if there exists  a triple  $(\alpha_V,\alpha_E,\alpha_L)$, where 
 $\alpha_V:V(\Gamma_1)\to V(\Gamma_2)$,  $\alpha_E:E(\Gamma_1)\to E(\Gamma_2)$
 and    $\alpha_L:L(\Gamma_1)\to L(\Gamma_2)$
are bijections such that $\alpha_V$   maps  the endpoints of $x\in E(\Gamma_1)\cup L(\Gamma_1)$
to the endpoints of $\alpha_E(x)$ or  $\alpha_L(x)$ for every $x\in E(\Gamma_1)\cup L(\Gamma_1)$.
Moreover $\forall v\in V(\Gamma_1)$ and $\forall e\in E(\Gamma_1)$ we have
$ 
w_1(v)=w_2(\alpha_V(v))
$ 
and 
$ 
\ell_1(e)=\ell_2(\alpha_E(e)).
$

\item The automorphism group
$\Aut(\Gamma,w)$ of  a weighted graph  $(\Gamma,w) $  
is given by triples $\alpha=(\alpha_V,\alpha_E,\alpha_L)$  
as in the previous item,   ignoring the condition on the length. 

\item  A weighted graph with $n$ legs, and hence  an $n$-pointed tropical curve, has finitely many
 automorphisms.
\end{enumerate}
\end{nota}

\begin{remark}
In the definition of $n$-pointed tropical curve we did not require that the points be  distinct, i.e. that the legs have different endpoints. This is because we shall work  modulo tropical equivalence, which does not
preserve this property. On the other hand,
 every tropical equivalence class of $n$-ponted curves contains representatives whose marked points are distinct; see \cite[Prop. 2.4.10]{CHBK}.
\end{remark}

The  addition of a weight function to  a   tropical curve (introduced in \cite{BMV}) is a way to fix  the fact that the set of pure tropical curves of given genus in not closed under specialization.
More precisely, families of tropical curves are given by letting the length of the edges vary. Now if some length goes to zero, it may very well happen that some cycle  gets contracted, and hence the first Betti number drops.  This problem does not arise when considering weighted tropical curves, as we are going to explain. 

First, let us formalize the process of edge length going to zero.
Fix a weighted graph  $(\Gamma,w)$, and  $S\subset E(\Gamma)$.
The {\it weighted contraction} of $S$ 
  is the weighted graph $(\Gamma/S,w/S)$,
 where 
  $\Gamma/S$  is defined in subsection~\ref{cont} in case $L(\Gamma)=\emptyset$;
  if $L(\Gamma)$ is not empty, the definition is trivially adjusted so that there is a natural identification
  between $L(\Gamma)$ and $L(\Gamma/S)$. 
To define  $w/S$, recall that we have a natural 
map $\sigma:\Gamma \to \Gamma/S$ and a natural surjection
 $\sigma_V:V(\Gamma)\to V(\Gamma/S)$.
 We set for every $\ov{v}\in V(\Gamma/S)$
 \begin{equation}
\label{wS}
 w/S(\ov{v})=b_1(\sigma^{-1}(\ov{v}))+\sum_{v\in \sigma_V^{-1}(\ov{v})}w(v).
  \end{equation}
 We write
\begin{equation}
\label{min}
(\Gamma,w) \geq (\Gamma' ,w')\  {\text{ if  }} \  (\Gamma' ,w'){\text{ is a weighted contraction of  }} (\Gamma,w).
\end{equation}
 
 \begin{remark}
Suppose $(\Gamma,w) \geq (\Gamma' ,w')$. 
Then one easily checks the following  properties
\begin{enumerate}
\item
$|L(\Gamma)|=|L(\Gamma')|$.
\item
$ 
g(\Gamma, w)=g(\Gamma', w') 
$ 
(by  identity \eqref{gendec}  and  remark~\ref{lme}).
\item
 If
  $(\Gamma,w)$ is stable, 
so is  $(\Gamma' ,w')$.
\end{enumerate}
 Therefore, the set of stable genus-$g$ graphs with $n$ legs is closed under weighted contractions.
\end{remark}

\subsection{The moduli space of pointed tropical curves}
\label{modss}
From now on we shall consider tropical curves up to tropical equivalence. Therefore we  will assume that our weighted graphs are stable.

Let us fix the stable graph $(\Gamma, w)$ with $n$ legs,
let $g=g(\Gamma, w)$, and let us consider the space $M(\Gamma,w)$
of isomorphism classes of
tropical curves having  $(\Gamma, w)$ as combinatorial type.
More precisely, we have a natural identification:
$$
M(\Gamma,w)=(\R_{>0})^{E(\Gamma)}/\Aut (\Gamma, w)
 $$
 where an automorphism $(\alpha_V,\alpha_E,\alpha_L)\in \Aut (\Gamma, w)$ acts by permuting the coordinates of 
 $(\R_{>0})^{E(\Gamma)}$ according to $\alpha_E$; see item (\ref{iso}).
In particular, $M(\Gamma,w)$ is an orbifold of dimension $|E(\Gamma)|$,
since $\Aut (\Gamma, w)$ is finite.
The set $ M(\Gamma,w)$ is thus a topological space, with the quotient topology induced by the euclidean topology. 

We recall the following well known and easy to prove fact:
\begin{remark}
\label{3r}
Let $(\Gamma, w)$ be a genus $g$ stable graph with $n$ legs.
Then  $|E(\Gamma)|\leq 3g-3+n$ and equality holds if and only if $\Gamma$ is a $3$-regular graph with
$b_1(\Gamma)=g$. Moreover, 
in this case we necessarily have
 $w=\underline{0}.$ \end{remark}

 We now introduce the moduli space, $\Mgtn$, of  $n$-pointed tropical curves of genus $g$:
\begin{equation}
\label{Mg}
\Mgtn=  \bigsqcup_{\stackrel{(\Gamma,w)  {\text{ stable}}}{\text{genus }g, \  n \text{ legs }}} {M(\Gamma,w)}.
 \end{equation}
The following statement is a summary of  some of the properties of $\Mgtn$ (see \cite{CHBK} for details;  in the   case $n=0$ some of the properties below are proved also in \cite{BMV}).
 \begin{fact}
  \label{Mgt}
Assume $2g-2+n>0$ and let $(\Gamma,w)$ be a stable graph of genus $g$ with $n$ legs.
  \begin{enumerate}
 \item
$\Mgtn$ is endowed with a topology such that  the  natural  injection  $M(\Gamma,w)\ha \Mgtn$  is a homeomorphism with its image.
\item
\label{corr2} With the notation   \eqref{min}, we have
$$
  M  (\Gamma', w')\subset\ov{M (\Gamma, w)}
\Leftrightarrow (\Gamma, w)\geq (\Gamma', w').$$
\item

  \label{Mgtr}
Let $\Mgrn\subset \Mgtn$ be the subset  parametrizing regular curves, i.e.
 $$
 \Mgrn =\bigsqcup_{\stackrel{|L(\Gamma)|=n, \  b_1(\Gamma)=g}{\Gamma \  \  3-{\text{regular}}}}{M(\Gamma,\underline{0})}\subset \Mgtn.
 $$
Then 
$\Mgrn$   is  open and dense in $\Mgtn$.

\item
  \label{Mgtp}
 Let $\Mgpnn$ be the subset  parametrizing pure tropical curves.  
Then  $\Mgpnn$   is   open and dense $\Mgtn$.

\item
  \label{MgtH}
 $\Mgtn$ is a connected, Hausdorff topological space of pure dimension $3g-3+n$.
 \end{enumerate}
\end{fact}
  
\begin{remark}
\label{dim}
We need to explain  the meaning of the last statement. 
Recall that   a topological space $X$ containing a dense open subset $U$, where $U$ is an orbifold 
 (locally the quotient of a topological manifold by a finite group) of dimension $d$,  is said to have     {\it pure dimension} $d$.  

Now, by part (\ref{Mgtr}), $\Mgtn$ contains the dense open subset $U=\Mgrn$, which is an orbifold has dimension $3g-3+n$,
by Remark~\ref{3r}. This explains the claim on the dimension. 
Connectedness of $\Mgtn$ is trivial, since   every $(\Gamma, w)$ satisfies
$(\Gamma, w)\geq (\Gamma^*, w^*)$ where $(\Gamma^*, w^*)$ is the graph having   no edges, only one vertex of weight $g$, and $n$ legs attached to it. 
 The fact that $\Mgtn$ is Hausdorff is proved  in   \cite[sect. 3.2]{CHBK}.
\end{remark}
 \subsection{Connectedness properties of tropical moduli spaces.}
\label{scho}
In this last subsection we apply our Linkage Theorem~\ref{main} to the geometry of some moduli
spaces of tropical curves.
 
 To begin with, we have said that $\Mgtn$ is connected; but a stronger form of connectedness holds,
 namely $\Mgtn$, and likewise $\Mgpnn$,  is  connected through codimension one; see Definition~\ref{conndef}. This property is one that is fundamental  for tropical varieties defined by prime ideals
(see \cite{MS}). Although $\Mgtn$ and $\Mgpnn$ are not known to be tropical varieties in general
(the case $g=0$ is a well known   exception),   their connectedness through codimension one
is   a sign of their being somewhat close to tropical varieties.

The next definition is adapted from  \cite[Definition 3.3.2]{MS}.
\begin{defi}
\label{conndef}
Let $X$ be a topological space of  pure dimension $d$;
see \ref{dim}.  
Assume that $X$ is endowed with a decomposition $X=\sqcup_{i\in I}X_i$,
where every $X_i$ is a connected orbifold.
 We say that $X$ is {\it connected through codimension one} if the subset
 $$
 \bigsqcup_{i\in I: \dim X_i\geq d-1}X_i\subset X
 $$
 is connected. 
\end{defi}

Notice that if $X$ is pure dimensional and connected through codimension one,
 then $X$ is connected.

Now, observe that the notion of linked graphs, given in Definition~\ref{link}, extends word for word to graphs with legs.
We can therefore state the following result, which is a consequence of Theorem~\ref{main}.

\begin{prop}
\label{linkn} Let $\Gamma_1$ and $\Gamma _2$ be  two  $3$-regular  graphs  with $n$ legs and 
$ b_1(\Gamma_1)=b_1(\Gamma_2)$.
Then $\Gamma_1$ and $\Gamma _2$ are linked.
\end{prop}
\begin{proof}
Of course, $|E(\Gamma_1)|=|E(\Gamma_2)|$; we can assume $|E(\Gamma_i)|\geq 2$ for otherwise the result is trivial.
We use induction on $n$; the base case $n=0$ is a special case of Theorem~\ref{main}.

Suppose $n\geq 1$; let us denote by $\mathcal{G}(n)$ the set of $3$-regular graphs of genus $g$ with $n$ legs. Let $\Gamma\in \mathcal{G}(n)$, pick a leg $l\in L(\Gamma)$ and let $v\in V(\G)$ be its endpoint.
Let $\Gamma'$ be the closure of the graph obtained by removing $l$ and $v$ from $\Gamma$.
It is clear that $\Gamma'\in \mathcal{G}(n-1)$.  Notice that, of course, every $\Gamma\in \mathcal{G}(n)$ is obtained  by adding a leg and its endpoint to
some graph in $\mathcal{G}(n-1)$. 

\noindent{\bf Claim.}
{\it Fix  a graph $\Gamma'\in \mathcal{G}(n-1)$; any two graphs in $\mathcal{G}(n)$  obtained by adding to $\Gamma'$ a leg and its endpoint  are linked.}

The claim implies our Proposition. Indeed, let $\Gamma'_1, \Gamma'_2\in  \mathcal{G}(n-1)$ be such that
for, some $e'_i\in E(\Gamma'_i)$ we have
\begin{equation}
\label{conteq'}
\Gamma'_1/e'_1=\Gamma_2'/e'_2.
\end{equation}
Let $\Gamma_1\in  \mathcal{G}(n)$ be obtained by adding to $\Gamma'_1$ a leg  whose endpoint is not in the interior of $e'_1$. Then, by \eqref{conteq'}, there exists a $\Gamma_2\in  \mathcal{G}(n)$ obtained by adding a leg and its endpoint to $\Gamma'_2$ such that $\Gamma_1/e_1=\Gamma_2/e_2$; so $\Gamma_1$ is linked to $\Gamma_2$. Hence, by the claim, we get that all graphs in $\mathcal{G}(n)$
obtained from $\Gamma'_1$ are linked to those obtained from $\Gamma'_2$. By the induction hypothesis
  every pair of elements in $\mathcal{G}(n-1)$
is linked, so  we are done.

It remains to prove the claim.  For $i=1,2$, let $\G_i\in \mathcal{G}(n)$ be the graph obtained by adding to $\Gamma'$ a vertex $v_i$ 
(in the interior of some edge or leg of $\G '$) and a leg $l_i$ adjacent to $v_i$.
We must show that $\G_1$ and $\G_2$ are linked.
Pick $w\in V(\Gamma')\subset V(\G_i)$;
  for $i=1,2$ the vertex $v_i$     can be joined to $w$ by some path $\Pi_i$ of minimal length contained in $\G_i$;
let  $h_i$ be the edge-length of $\Pi_i$, where $h_i$ is a positive integer, since $w\neq v_i$; we call $h_i$   the   edge-path length from $v_i$ to $w$.
Let $h=h_1+h_2$; if $h=2$, i.e. if $h_1=h_2=1$, there exists an edge $e_i\in E(\Gamma_i)$
whose endpoints are $w$ and $v_i$. It is clear that 
$$
\Gamma_1/e_1=\Gamma_2/e_2
$$
so we are done. We continue by induction on $h$.

Suppose $h \geq 3$, and let $h_1\geq 2$. 
Let $e_1$ be the first edge of $\Pi_1$  so that  $v_1$ is an   endpoint of $e_1$;
 Consider the graph  $\Gamma_1/e_1$. The coming construction is illustrated in the picture below.
Now let $u$ be the other endpoint of $e_1$ and let $f$ be the next edge of $\Pi_1$, starting at $u$; by construction, $f$ is also an edge of $\Gamma'$.
Let $\Gamma_3\in \mathcal{G}(n)$ be the graph obtained from $\Gamma'$ by adding a vertex $v_3$ in the interior of $f$ and a leg attached to it. Now, $\Gamma_3$ has a unique edge $e_3$ whose endpoints are $u$ and $v_3$. The edge-path length from $v_3$ to $w$ is $h_1-1$, hence by induction $\Gamma_3$ is linked to $\Gamma_2$.
On the other hand it is immediately clear that 
  $$
\Gamma_1/e_1=\Gamma_3/e_3,
$$
hence  $\Gamma_3$ and $\Gamma _1$ are linked, and so $\Gamma _1$ is also linked to $\Gamma_2$ .
\end{proof}
The next picture represents the construction used to prove the claim, with $g=2$ and $n=3$..
\begin{figure}[h]
\label{1fig}
\begin{equation*}
\xymatrix@=.5pc{
&&&&&&&&&&&&&&&&&&&&&&&&&&&&&&&&&&&\\
&&&&\Gamma'=&&*{\bullet} \ar@{-}[ul]\ar@{-}[dl]\ar@{-}[rrr]&&&*{\bullet}\ar@{-}@/_1.9pc/[rr] \ar@{-}[rr]^<{u}_(.5){f}
&& *{\bullet}\ar@{-}[r]^>{w}& *{\bullet}\ar@{-}@(ur,dr)&&&\\
&&&&&&&&&&&&&&&&&&&&&&&&&&&&&&&&&&&&\\
&&&&&&&&&&&&&&&&&&&&&&&&&&&&&&&&\\
&&&&&&&&&&&&&&&&&&&&&&&&&&&&&&&&\\
&&&&&&&&&&&&&&&&&&&&&&&&&&&&&&&&&&&&\\
\Gamma_1=&&*{\bullet} \ar@{-}[ul]\ar@{-}[dl]\ar@{-}[r]&*{\bullet}\ar@{-}[r]^<{v_1}_>(.6){e_1}
\ar@{-}[dl]^(.7){l_1}&*{\bullet}\ar @{-} @/_1.9pc/[rr] \ar@{-}[rr]^<{u}_(.5){f}
&&*{\bullet}\ar@{-}[r]^>{w}& *{\bullet}\ar@{-}@(ur,dr)&&&&\Gamma_3=& 
&*{\bullet} \ar@{-}[ul]\ar@{-}[dl]\ar@{-}[rr]&&*{\bullet}\ar@{-}@/_1.9pc/[rrr] \ar@{-}[rr]^<{u}_>(.6){e_3}&&*{\bullet}\ar@{-}[r]^<(.3){v_3}\ar@{-}[ul]_(.8){l_3}
&*{\bullet}\ar@{-}[r]^>{w}& *{\bullet}\ar@{-}@(ur,dr)&&&&\\
&&&&&&&&&&&&&&&&&&&&&&&&&&&&&&&&&&&\\
}
\end{equation*}
\caption{$\Gamma_1$ and $\Gamma_3$ linked, obtained from $\Gamma'$ (proof of \ref{linkn}).} 
\end{figure}
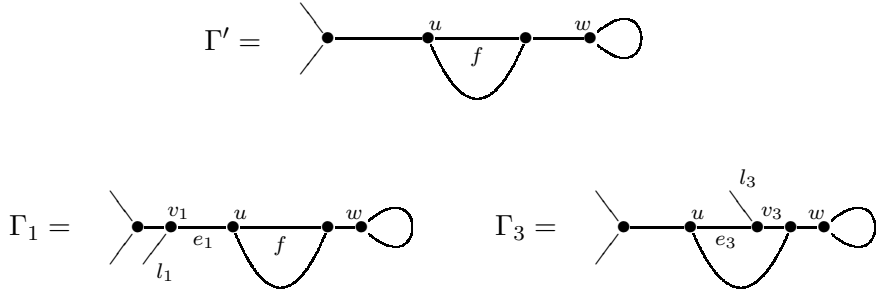
\

From Proposition~\ref{linkn} we easily get:
\begin{prop}
\label{connt1}
The spaces $\Mgtn$ and $\Mgpnn$ are connected  through codimension one.
\end{prop}
\begin{proof}
From Fact~\ref{Mgt} we have that $\Mgtn$ and $\Mgpnn$ are of pure dimension $3g-3+n$.
Also, we know that $\dim  M(\Gamma, w) =|E(\Gamma)|$.
So, by \ref{Mgt} (\ref{corr2})
to prove our statement it suffices to observe that any two $3$-regular graphs are linked,
as stated in Proposition~\ref{linkn}.
\end{proof}

\begin{remark}
As proved in \cite[Prop. 3.2.5]{BMV}, the above result in case $n=0$ follows from \cite[Prop. page 236]{HT}, which is a remarkable and well
known   special case of our Theorem~\ref{main}. 
\end{remark}
\begin{nota}{\it The tropical Torelli map and the Schottky locus.}
We will now prove that   connectedness  through codimension
one holds for other tropical moduli spaces.

In analogy with the classical situation we have a tropical Torelli map
$$
\tgt:\Mgt \to \Agt
$$
to the moduli space of tropical Abelian varieties, mapping a curve to its tropical Jacobian  (see \cite{MZ}, \cite{CV} and \cite{BMV} for details). 
We denote by $\Sgt$ the image of $\tgt$, and refer to it, as it is customary,
as the tropical Schotty locus in $\Agt$. 
A detailed  analysis of $\Sgt$ for small values of $g$ is carried out in \cite{chan}. 

For our purposes 
  $\Sgt$ can be identified with the topological quotient
$$
\Sgt:=\Mgt/\equiv_{\tgt}
$$
where $[(\Gamma, \ell ,w)]\equiv_{\tgt}[(\Gamma, \ell ,w)] \Leftrightarrow \tgt([(\Gamma, \ell ,w)])=\tgt([(\Gamma', \ell ',w')]).$
For more  structure on $\Agt$ and $\Sgt$
we refer to \cite{BMV}.  In particular, Theorem 5.2.4 of loc. cit. gives a precise characterization of the tropical Schotty locus $\Sgt$ in $\Agt$, in such a way that  the Schottky problem has a satisfactory answer    in tropical geometry.

As proved in \cite[Thm 4.1.9]{CV},
and generalized by \cite[Thm 5.3.3]{BMV}, 
the Torelli map identifies curves having the same so-called ``3-edge-connected class''.
More precisely,   let us denote by $\Mge$ the locus
of tropical curves with 3-edge-connected graph:
$$
\Mgt \supset  \Mge:=\{[(\Gamma, \ell, w)]:\  \Gamma   {\text { is 3-edge-connected}}\}.
$$
Then
we have
$$
\tgt(\Mge)=\tgt(\Mgt)= \Sgt\subset \Agt.
$$
Furthermore,  the restriction  of $\tgt$   to $\Mge$, denoted by $\tgt[3]$,  is  injective on every subspace 
  $M(\Gamma,w)\subset \Mge$, and it identifies  two such spaces, $M(\Gamma,w)$ and $M(\Gamma',w')$,
only if the graphs $\Gamma$ and $\Gamma'$ are cyclically equivalent
(i.e. 2-isomorphic in the sense of Whitney, see \cite[Def 2.2.3]{CV}). In particular,  $\tgt[3]$ has finite fibers.

The previous results hold  in the special case of pure tropical curves
(in fact, they were first proved in this case, and then generalized to weighted tropical curves).
With self-explanatory notation, 
the  Torelli map for pure tropical curves is a surjection
$$
\tgp:\Mgp \la \Sgp:=\Mgp/\equiv_{\tgp}\subset\Agt,
$$
and the restriction of $\tgp$ to the locus of pure tropical curves with 3-edge connected graph,
$\Mgpe\subset \Mgp$, behaves exactly as $\tgt[3]$.

Now, the conservation of 3-edge-connectivity under linkage, proved in Theorem~\ref{main}, enables us to obtain the following result.
\begin{thm}
\label{conn1}
The spaces $\Mge$ and  $\Sgt$  have pure dimension equal to $3g-3$ and are
connected through codimension one.

The same holds for the spaces $\Mgpe$ and  $\Sgp$.
\end{thm}
\begin{proof}
We prove the result for tropical curves; the proof for pure tropical curves follows precisely  the same lines (and it is actually simpler).
We
introduce the locus  of  regular,  3-edge-connected  curves
$$\Mgre\subset \Mgr\subset \Mgt.$$  
We have that the closure in $\Mgt$ of regular,  3-edge-connected curves is 
the locus of all 3-edge-connected curves,
i.e.
$$
\ov{\Mgre}=\Mge.
$$
This follows from \cite[Prop A.2.4]{CV}, whose proof (stated there only for pure regular curves) works also   in our setting (i.e. for weighted  tropical curves).
It is clear that $\Mgre$ is an orbifold of pure dimension $3g-3$. We conclude that  
   $\Mge$ has pure dimension $3g-3$.

Now, the   connectedness through codimension one follows from the part of Theorem~\ref{main}
concerning 3-edge connected graphs. It suffices to 
add that if $(\Gamma',w')$ is obtained from $(\Gamma,w)$ by contracting only one edge,
then $\dim M(\Gamma,w)=\dim M(\Gamma',w')+1$ (as we also did  for Proposition~\ref{conn1}).
This proves that $\Mge$ is connected through codimension one.

Now we turn to the Schottky locus; by what we said before there is a surjection with finite fibers 
$$
\tgt[3]:\Mge\la \Sgt  
$$ 
obtained by restricting the Torelli map.
This surjection
induces a homeomorphism with its image of every subspace $M(\Gamma,{\underline{0}})\subset \Mge$.
This implies that $\Sgt$ has pure dimension $3g-3$.
Furthermore, as $\tgt[3]$ is injective on every   $M(\Gamma,w)$, it preserves the dimension
of these subsets; therefore $\Sgt$ is connected through codimension one, because so is $\Mge$.
\end{proof}

\begin{remark}
What are the consequences on tropical moduli spaces  of the linkage theorem when $p\geq 4$?
Consider the subset
$$
M_g^{p{\rm{-reg}}}:=\bigsqcup_{\stackrel{\Gamma \  p{\text {-regular}}}{b_1(\Gamma)=g}}M(\Gamma, {\underline{0}}) \subset \Mgp
$$
and assume it is not empty.
By  a proof similar to that of Theorem~\ref{conn1} one obtains that
the closure   of $M_g^{p{\rm{-reg}}}$
is of pure dimension equal to $p(g-1)/(p-2)$, by Remark~\ref{count}
(this number is   an integer  by the non-emptyness assumption), and  
connected through codimension one.

The same holds if the above disjoint union is restricted to all 3-edge-connected and $p$-regular
graphs with $b_1(\Gamma)=g$. That is, with self-explanatory notation, the closure of $M_g^{p{\rm{-reg}}}[3]$ is of pure dimension   $p(g-1)/(p-2)$   and 
connected through codimension one.
\end{remark}

A space closely related  to $\Mgt$ is the outer space $O_g$ constructed in \cite{CuVo}, and its quotient by
the group $\Out(F_g)$ (outer automorphisms of the free group on $g$ generators $F_g$).
This quotient can be interpreted as a moduli space for metric graphs, 
and its connection with $\Mgt$ or $\Mgp$    is currently under investigation;
as it has not yet been completely unraveled, we will not be more specific about this point.
We just wish to   mention that Theorem~\ref{main} applied to $O_g$ yields analogous connectivity properties of certain subcomplexes of a remarkable deformation retract of $O_g$, called its ``spine" (defined in \cite[sect 1.1]{CuVo}).

\end{nota}

\end{document}